\documentclass[11pt]{amsart}       
\usepackage{graphicx}

\usepackage{bigints}
\usepackage{todonotes}
\presetkeys{todonotes}{color=orange!30}{}
\usepackage{color}

\usepackage{mathrsfs} 
\usepackage{eucal}	
\usepackage{amssymb}

\usepackage{tikz} 
\usepackage{amsmath} 
\usepackage{cancel}
\usepackage{enumitem} 

\usepackage{hyperref} 
\usepackage{varioref} 

\usepackage{url}

\usepackage{color}

\def\mmR{\textcolor{violet}} 

\newtheorem{theorem}{Theorem}[section]
\newtheorem{lemma}[theorem]{Lemma}
\newtheorem{corollary}[theorem]{Corollary}

\newtheorem{definition}[theorem]{Definition}
\newtheorem{example}[theorem]{Example}

\newtheorem{remark}{Remark}

\newcommand{\numberset}{\mathbb} 
\newcommand{\R}{\numberset{R}}
\newcommand{\N}{\numberset{N}}
\newcommand{\Z}{\numberset{Z}}

\newcommand{\beautyset}{\mathcal} 

\newcommand{\M}{\beautyset{M}}
\newcommand{\T}{\beautyset{T}}

\newcommand{\F}{\beautyset{F}}

\newcommand{\J}{\mathbf{j}}
\newcommand{\V}{\beautyset{V}}

\newcommand{\B}{\beautyset{B}}

\newcommand{\del}{\mathfrak{d}}

\def\Ss{{\bf t}}

\def\u{{\bf u}}

\def \ds{\displaystyle}

\def \vsm{\vskip 0.2 truecm}
\def \vsmm{\vskip 0.1 truecm}

\def\bel{\begin{equation}\label}
\def\eeq{\end{equation}}
\def \w{\omega}
\def \d{{\bf d}}

\def\sgn{\text{{\rm sgn}}}

\begin{document}

\title[HJ involving Lie brackets and stabilizability with a cost]{HJ inequalities involving  Lie brackets   and \\  feedback  stabilizability  with  cost regulation}
\author{Giovanni Fusco}\address{G. Fusco, Dipartimento di Matematica,
Universit\`a di Padova\\ Via Trieste, 63, Padova  35121, Italy\\
email:\,
fusco@math.unipd.it }
\author{Monica Motta}\address{M. Motta, Dipartimento di Matematica,
Universit\`a di Padova\\ Via Trieste, 63, Padova  35121, Italy\\
email:\,
motta@math.unipd.it}
\author{Franco Rampazzo}\address{F. Rampazzo, Dipartimento di Matematica,
Universit\`a di Padova\\ Via Trieste, 63, Padova  35121, Italy\\
email:\,
rampazzo@math.unipd.it}

\maketitle
\begin{abstract}
With reference to  an optimal control problem where the state has to approach asymptotically a closed target  while paying a non-negative  integral cost, we propose a   generalization of the classical dissipative  relation that defines  a Control Lyapunov Function to a weaker differential  inequality. The latter  involves both the cost and the   iterated  Lie brackets of the vector fields in the dynamics up to a certain degree $k\geq 1$, and we call any of its  (suitably defined) solutions a {\it degree-$k$ Minimum Restraint Function}. We prove that the existence of    a  degree-$k$ Minimum Restraint Function allows us to build a   Lie-bracket-based feedback which sample stabilizes the system to the target while {\it regulating} (i.e., uniformly bounding) the cost. 
\end{abstract}

 	

\subsection*{Keywords} \texttt{Lyapunov functions, asymptotic stabilizability,  discontinuous feedback law, optimal control, Lie brackets. }
\subsection*{AMS Subject classification codes}\texttt{93D30, 93D20, 49N35, 34A26, 34D20}

\section{Introduction}
The stabilizability to a  target $\T\subset\R^n$ of  a control system $\dot y = f(y,\alpha)$ on $\R^n$  consists in the  existence of   a   feedback  law
	$\alpha:\R^n\setminus\T\to A$, where $A$ is the set of control values, whose (suitably defined) implementation  drives the state trajectory towards the target $\T$, in  some  uniform way. 
	Though stabilizability is a   dynamical issue, it  may obviously be considered together with a concomitant optimization problem.
	Specifically,  starting from any state $x\in\R^n\backslash\T$, one might wish to minimize a cost functional of the form 
	\bel{cost_intro}
	\int_0^{S_y} l(y(s),\alpha(s)) \, ds, \qquad{(l\ge0)}, 
	\eeq 
	over all  control-trajectory pairs $(\alpha, y):[0,S_y[\to A\times \R^n$ of  the control system \bel{sys_intro}\dot y = f(y,\alpha),\,\,\,\, y(0)=x,\eeq  where $0<S_y\le\infty$ denotes  the infimum among the times   $S$ that verify\linebreak $\ds\lim_{s\to S^-} {\rm dist}(y(s),\T)=0 
	$.  A natural question then arises: 
	\vsmm
\noindent ({\bf Q})  	{\it   Can we establish sufficient conditions  --in terms of existence of a special kind of Control Lyapunov Function--  to obtain a somehow extended notion  of stabilizability,  so to  include a simultaneous bound on the cost paid by the chosen control-trajectory pairs?}
	\vsmm
	As for  global asymptotic  controllability,  which is stabilizability's corresponding open loop notion, 
	early answers to   ($\mathbf{Q}$)  have been provided in  \cite{MR,LMR} by means of   the notion of Minimum Restraint Function (MRF). The latter is a particular  Control Lyapunov Function --i.e. a solution of a suitable Hamilton-Jacobi  dissipative  inequality-- which, besides implying  global asymptotic  controllability,   {\em regulates} the cost, namely   it provides a criterium for selecting control-trajectory pairs $s\mapsto(\alpha(s),y(s))$ which satisfy
	a   uniform  bound on the corresponding cost. 
Question ({\bf Q}) has received  some answers for  both bounded and unbounded controls   in \cite{LM} and \cite{LM2,LM19}, respectively,  where one shows how  a stabilizing feedback law,   which also ensures an upper bound on the cost, can be built  starting from a MRF, through a `sample and hold' approach. (Among the different notions of stabilizability , see e.g. \cite{SS80,Art83,Brockett,So,CLSS,TT,AB}, we  consider here the so-called {\it sample stabilizability}, see e.g. \cite{CLRS,R2,C10}.)

   In relation to 
the driftless 
control-affine 
     system  
	\bel{c_s_intro}
	\dot y(s) = \sum_{i=1}^m  f_i(y(s))\, \alpha^i(s), \qquad 
	\eeq
where the controls $\alpha$ take values in $A:=\{\pm e_1,\dots,\pm e_m\}$ ($e_i$ denoting the $i$th element of the   canonical basis of $\R^m$),  in this paper we aim to improve these results by  constructing  a   
Lie-bracket-determined stabilizing feedback which, at the same time, induces  a bound for cost \eqref{cost_intro}.    
	  This means, in particular, that  the utilized `sample and hold' technique  involves  not only  the  vector fields $f_1,\ldots,f_m$  but also  their iterated Lie  brackets.
%
%
	More precisely, we 	will consider  the \textit{degree-$k$ Hamiltonian} 
	\bel{hamk}
	H[p_0]^{(k)}(x,p,u) := \min_{B}  \Big\{\langle p ,  B (x) \rangle+ p_0(u) \ds \,\max_{a\in A_ B}l(x,a)\Big\},
	\eeq
	where $p_0:\R_{\ge 0}\to [0,1]$ is a continuous increasing function, here  called  {\it cost multiplier}.   The 
 minimum is  taken among  (signed)  iterated Lie brackets $B$  of the vector fields $f_1,\dots, f_m$   of length $\leq k$, while, for any bracket $B$, the inner maximization   is performed over the subset  $A_ B\subset A$ of control values
	  utilized to approximate the  $B$-flow. The novelty with respect to the standard Hamiltonian consists in the fact 
	that, {on the one hand}, for $k\geq 2$, $H[p_0]^{(k)}$ is obtained by minimization over vectors (the Lie brackets) which may not belong     to  the dynamics of the system, and,  on the other hand,   $H[p_0]^{(k)}$ depends also on the functions $l$ and $p_0$.
	Under some mild  hypotheses specified in Section  \ref{section_mrf}, the degree-$k$ Hamiltonian  $H[p_0]^{(k)}$  is well defined and continuous.  Finally,  in order to choose the  right  iterated  Lie bracket (in the construction of the feedback) we exploit the following  differential inequality:
	\bel{dissipative0intro} \begin{array}{l} 
	\ds	H[p_0]^{(k)}(x,p, U(x))  \leq -\gamma(U(x)) 
		\qquad \forall x \in \R^n\setminus \T,  \ 
	 \  \forall p\in\partial_{\mathbf{P}}U(x). 
	\end{array}
	\eeq
Here,  the {\it dissipative rate}	$\gamma$ is an increasing function taking values in $]0,+\infty[$ and $\partial_{\mathbf{P}} U(x)$  denotes   the proximal subdifferential of $U$ at $x$ (see \eqref{proximal}).
	We call relation \eqref{dissipative0intro}  the    {\em degree-$k$  HJ dissipative inequality} (where {\em HJ  } stands for Hamilton-Jacobi).   
	A proper,  positive definite, and continuous function   $U: \overline{\R^n \setminus \T} \to \R$  satisfying   \eqref{dissipative0intro} for some $p_0$ and $\gamma$,  is  called  a  \textit{degree-$k$  Minimum Restraint Function}  (in short,  \textit{degree-$k$ MRF}).  
	Let us observe that, as a trivial consequence  of the monotonicity
	$$
	H[p_0]^{(k)} \leq H[p_0]^{(k-1)} \leq \dots \leq H[p_0]^{(1)},
	$$ the higher the Hamiltonian's degree  the larger   the corresponding set of  MRFs.
 Furthermore, as shown in  an example in \cite{MR3}, for $k$ sufficiently large it can well happen that  a smooth 
	{\it degree}-$k$  MRF   does exist  while a standard, i.e. degree-1, $C^1$ MRF, does not  (see also  \cite[Ex. 2.1-2.4]{MR2},  for the case with no cost).  An increased regularity of a degree-$k$  MRF --possibly obtained by choosing a sufficiently large degree $k$-- is of obvious interest, both  from a numerical point of view and  in relation with feedback insensitivity   to data errors.

	Our  main result, which  will be  rigorously stated in   Theorem  \ref{Th_strong}, can be  roughly summarized   as follows:  
	\vsmm
 {\bf Main result.} 
	{\it Under some regularity and  integrability  assumptions, the existence of a  degree-$k$ Minimum Restraint Function $U:\overline{\R^n\setminus \T} \to \R$ 
		implies that    control system \eqref{c_s_intro} is degree-$k$ sample stabilizable to $\T$ with   regulated cost.} 
	\vsmm
	 	
  Let us observe  that the use of Lie brackets is a  well-established, basic tool in the investigations of   necessary   conditions for optimality as well as sufficient  conditions for  controllability (see e.g.\cite{K,Su,AMR20,AgSa,Co1}). Furthermore,  Lie algebraic  assumptions plays a crucial role in the study
 	of regularity and uniqueness  for  boundary value problems of  Hamilton-Jacobi equations  (see e.g. \cite{BCD} and references therein). However, in the mentioned literature   Lie brackets are not involved explicitly in the connected HJ equations inequalities, as instead is the case in our degree-$k$  dissipative inequality  \eqref{dissipative0intro}.  
	 	
    As for the   assumptions in the above statement, they 
	include some integrability properties of the cost multiplier $p_0$ and of  the dissipative rate $\gamma$.   Furthermore,
	 whenever $k>1$,  they also involve a certain interplay between the curvature parameters of $\partial \T$ and the  semiconcavity  coefficient
	  of the MRF $U$ (see Subsection \ref{mainre}). As a matter of fact, the need for  such a condition was somehow predictable,  since it results unavoidable in  the  high order controllability conditions for the minimum time problem    with a general target  (see e.g.  \cite{KrasQ,Kras,Ma,MaRi}), where  $U$ is the distance $\d$ from the target and $l\equiv1$ (see Rem. \ref{RTmin}).  
	For instance, when $\d$   is a solution of \eqref{dissipative0intro} for $p_0$ and $\gamma$  positive constants,   the  `regularity and integrability assumptions' of the Main Result are always satisfied as soon as either $\T$ is a singleton or $\T$ satisfies the internal sphere condition (see Rem. \ref{Rsc0}).

   Let us summarize the previous considerations by saying that,  on the one hand, our results extend previous results on  sample stabilizability by considering both higher order conditions and a regulated cost. Also,  our main result extends  some achievement of \cite{MR2,Fu}, which concerned the case without a cost. On the other hand,  it  can  be regarded as a generalization to  problems  with a nonnegative vanishing Lagrangian $l$ (widely investigated from a PDEs' point of view, see e.g. \cite{Mal,MS1,MS2}) of both classical  achievements   on H\"older continuity and  more recent regularity results    for the  minimum time problem. 

The paper is organized as follows.   Section  \ref{section_mrf}  is devoted to  state the definitions of   degree-k feedback generator and degree-k sample stabilizability with regulated cost.  In Section  \ref{sec_results} we introduce the  precise assumptions and  establish the main result, whose proof is given in Section   \ref{proof_section}. 
	In  Section \ref{Sec_gen} we briefly mention possible generalizations of both the form of the dynamics (no longer driftless control affine) and in the regularity assumptions. In particular, the latter can be weakened up to the point of considering set-valued Lie brackets of Lipschitz continuous vector fields.  

\subsection{Notation and preliminaries} \label{preliminari}
For any  $a,b\in\R$, let us set  $a\vee b:= \max\{a,b\}$, $a\wedge b:= \min\{a,b\}$. For any integer $N\ge1$, $(\R^N)^*$ denotes  the dual space of $\R^N$ (and is often identified with $\R^N$ itself), while  we set $\R_{\ge0}^N:=[0,+\infty[^N$ and $\R_{>0}^N:=]0,+\infty[^N$. For  $N=1$ we simply write $\R_{\ge0}$ and  $\R_{>0}$, respectively. We denote the closed unit ball in $\R^N$    by  $\B_N$ and, given $r>0$, $r\B_{N}$  will stand for the  ball of radius $r$  (we do not specify the dimension when it is clear from the context). 
Given two nonempty sets $X$, $Y \subseteq \R^N$, we call {\em distance} between $X$ and $Y$  the number ${\rm dist}(X, Y) := \inf\{ |x-y| \text{ : } x\in X \text{, } y \in Y\}$. We do observe that ${\rm dist}(\cdot,\cdot)$ is not a distance in general, while the map $\R^N\ni x\mapsto {\rm dist}(\{x\},X)$ coincides with the distance function from $X$. 
We set $\B(X,r) := \{ x \in \R^N \text{ : } {\rm dist}(\{x\},X) \leq r \}$
and we write $\partial X$, int$(X)$, and $\overline X$ for the boundary, the interior, and the closure of $X$, respectively. 
 For any two points $x$, $y \in \R^N$ we denote by sgm$(x,y)$ the segment joining them, i.e. sgm$(x,y) := \{ \lambda x+(1-\lambda)y \text{ : } \lambda\in[0,1] \}$. Moreover, for any two vectors $v,w \in \R^N$, $\langle v,w\rangle$  denotes their scalar product.
\vsm 
Let   $\Omega \subseteq  \R^N$ be an open, nonempty subset. 
 Given an integer $k\geq 1$, we write $C^k(\Omega)$ for the set of vector fields of class $C^k$ on $\Omega$, namely $C^k(\Omega):=C^k(\Omega; \R^N)$, while  $C^{k}_b(\Omega) \subset C^k(\Omega)$ denotes the subset
 of vector fields with bounded derivatives up to order $k$.
We use  $C^{k-1,1}(\Omega)\subset C^{k-1}(\Omega)$ to denote the subset of vector fields whose $(k-1)$-th derivative is Lipschitz continuous on $\Omega$, and we set  and $C^{k-1,1}_b (\Omega):=C^{k-1}_b(\Omega)\cap C^{k-1,1}(\Omega) $. 


\noindent  We say that a continuous function $G:\overline{\Omega} \to\R$ is  {\em positive definite} if  $G(x)>0$ \,$\forall x\in\Omega$ and $G(x)=0$ \,$\forall x\in\partial\Omega$. The function $G$ is called {\em proper}  if the pre-image $G^{-1}({\mathcal K})$ of any compact set ${\mathcal K}\subset \R$ is compact.   
  We use     $\partial_{\mathbf{P}}G(x)$  to denote  the (possibly empty) {\em proximal subdifferential of $G$ at $x$}, namely the subset of  $(\R^N)^*$  such that,  for some positive constants  $\rho$, $c$, one has
 \bel{proximal}
p\in \partial_{\mathbf{P}}G(x)\iff G(\bar x)-G(x)+\rho|\bar x-x|^2\ge \langle p\,,\, \bar x-x\rangle \qquad \forall \bar x\in  \B(\{x\}, c).
\eeq
The {\em limiting subdifferential  $\partial G(x)$ of $G$ at $x\in \Omega$,}  is defined as
\[ 
 \partial G(x) := \Big\{p\in\R^N: \ \ p=\lim_{i\to+\infty}, \   p_i: \  p_i\in \partial_{\mathbf{P}}G(x_i), \ \lim_{i\to+\infty} x_i=x\Big\}.
\]
 When the function $G$ is locally Lipschitz continuous on $\Omega$, the limiting subdifferential $\partial G(x)$ is nonempty at every point.  

\noindent We say that a function $G: \Omega \to \R$ is \textit{semiconcave} (with linear modulus) on $\Omega$ if it is continuous and for any closed subset $\M\subset\Omega$ there exists $\eta_{_\M} >0$ such that  
\[  G(x_1) + G(x_2) - 2G\left( \frac{x_1+x_2}{2}\right) \leq \eta_{_\M} |x_1-x_2|^2      \]
for all $x_1$, $x_2 \in \M$ such that the segment sgm$(x_1,x_2)$ is contained in $\M$. If this property is valid just for any compact subset $\M\subset\Omega$, $G$ is said to be \textit{locally semiconcave} (with linear modulus) on $\Omega$.
  In both cases, $G$ is locally Lipschitz continuous and,
for $\M$, $\eta_{_\M}$, $x_1$, $x_2$ as above,
  for any  $p \in \partial G(x)$ the following inequality  holds true  as well
(see \cite[Prop. 3.3.1, 3.6.2]{CS}):
\bel{semiconcave_est}
 G(x_2) - G(x_1) \leq \langle p , x_2-x_1 \rangle + \eta_{_\M} |x_2-x_1|^2 .
\eeq

\vsm 
Finally, let us collect some basic definitions on iterated Lie brackets.  If $g_1$, $g_2$ are $C^1$ vector fields on $\R^N$  the {\it Lie bracket of $g_1$ and $g_2$} is defined  as
$$
[g_1,g_2](x) := Dg_2(x)\cdot g_1(x) -  D g_1(x)\cdot g_2(x)\,\,\, \big(= - [g_2,g_1](x)\big).
$$ 
As is well-known, the map $[g_1,g_2]$ is a true
 vector field, i.e. it can be defined intrinsically. If the vector fields are sufficiently regular,  one can iterate the bracketing process: for instance, given a $4$-tuple ${\bf g}:=(g_1,g_2,g_3,g_4)$ of vector fields  one can construct  the brackets $[[g_1,g_2],g_3]$,  $[[g_1,g_2],[g_3,g_4]]$,  $[[[g_1,g_2],g_3],g_4]$, $[[g_2,g_3],g_4]$.   Accordingly, one can consider the  {\it  (iterated)  formal brackets } $B_1:=[[X_1,X_2],X_3]$, $B_2:=[[X_1,X_2],[X_3,X_4]]$, $B_3:=[[[X_1,X_2],X_3],X_4]$, $B_4:=[[X_2,X_3],X_4
 ]$ (regarded as  sequence of letters $X_1,\ldots,X_4$, commas, and left and right square  parentheses), so that, with obvious meaning of the notation,   $B_1({\bf g}) = [[g_1,g_2],g_3]$,   $B_2({\bf g}) = [[g_1,g_2],[g_3,g_4]]$, $B_3({\bf g}) =[[[g_1,g_2],g_3],g_4]$, $B_4({\bf g}) =[[g_2,g_3],g_4]$.

 \noindent The {\it degree} (or {\it length}) of a  formal bracket is   the number $\ell_{_B}$ of letters that are   involved in it.  For instance, the brackets $B_1, B_2, B_3, B_4$ have degrees equal to  $3$, $4$, $4$, and $3$, respectively.  By convention,  a single  variable $X_i$  is a formal bracket of degree $1$. {Given a formal bracket $B$ of degree $\ge2$, then there exist formal brackets $B_1$ and $B_2$ such that $B=[B_1,B_2]$. The pair $(B_1,B_2)$ is univocally determined and it is called  the {\em factorization} of $B$.}

\noindent The  {\it switch-number}  of a formal  bracket  $B$ is the number $\mathfrak{s}_{_B}$ defined recursively  as:  
\[ 
\mathfrak{s}_{_B} := 1 \  \text{ if $\ell_{B}=1$}, \qquad  
\mathfrak{s}_{_{B}}:= 2\big(\mathfrak{s}_{_{B_1}}+\mathfrak{s}_{_{B_2}}\big)\  \text{ \ if  \ $\ell_B\geq2$ \ and  \ ${B}=[B_1,B_2]$.}
\] 
\noindent For instance, the switch-numbers of $[[X_3,X_4],[[X_5,X_6],X_7]]$ and $[[X_5,X_6],X_7]$ are $28$ and $10$, respectively. When  no  confusion may arise,  we  also speak of  `degree and  switch-number of  Lie brackets of vector fields'.
It may happen that brackets with the same degree have different switch numbers. However, if we set,  for any integer $k\geq1$,
\bel{bk}
\begin{cases}
	\beta(k) = 2[\beta(k-1)+1] \quad \text{if }k\geq2, \\
	\beta(1) = 1,
\end{cases}
\eeq and if $B$ is a formal bracket with degree $\ell_{_B}\leq k$, then 
$\mathfrak{s}_{_B} \leq \beta(k)$. 

\noindent We will use the following notion of {\em admissible bracket pair}:

 \begin{definition}\label{classCB}    Let $c\geq 0$,  $\ell\ge 1$, $q\geq c+\ell$ be integers, let  $B = B(X_{c+1},\ldots,X_{c+\ell})$ be  an iterated  formal bracket  and let   ${\bf g}=(g_1,\ldots,g_q )$  be a string of continuous vector fields. We say that  ${\bf g}$  is {\it of class $C^B$}   if  there exist non-negative integers $k_1\dots,k_q$ such that, by the only information that $g_i$ is of class  $C^{k_i}$ for every $i=1,\dots,q$  , one can deduce that 
$B({\bf g})$ is a  $C^0$ vector field  (see \cite[Def. 2.6]{FR2}). In this case,  we call $(B,{\bf g})$ an {\em admissible bracket pair} (of degree $\ell$ and switch number $\mathfrak{s}:=\mathfrak{s}_{_B}$).
  \end{definition}
 
 \noindent  For instance, if   $B= [ [[X_3,X_4],[X_5,X_6]],X_7 ]$ and  ${\bf g}=  
  (g_1,g_2,g_3,g_4,g_5,g_6,g_7,g_8)$,  then ${\bf g}$  is of class $C^B$ provided   $g_3,g_4,g_5,g_6$ are of class  $C^{3}$ and $g_7\in C^{1}$.

\section{Degree-$k$  sample stabilizability  with regulated cost} \label{section_mrf}
 Let us recall from  \cite{FMR1}  the  definitions of    degree-$k$ feedback generator,  sampling process,   and degree-$k$ sample stabilizability with regulated cost. 
 \vsm
 Throughout the whole paper $k\ge1$ will be a given integer and
we will  consider the following sets of hypotheses:
{\em  
\begin{itemize}
\item[{\bf (H1)}]
The control set  $A\subset \R^m$ is defined as  $A:=\{\pm e_1, \dots, \pm e_m\}$ and 
the target  $\T\subset\R^n$ is a closed subset with compact boundary.
 \vsm
\item[{\bf (H2)}]   the Lagrangian $l:\R^n\times A\to \R_{\ge0}$  is such that,  for any $a\in A$, the function $\R^n\ni x\mapsto l(x,a)$, is locally  Lipschitz continuous. Furthermore, the  vector fields $f_1,\dots, f_m:\R^n\to\R^n$ belong to   $C^{k-1,1}_{b}(\Omega)$ for any bounded, nonempty subset $\Omega\subset\R^n$.  
 \end{itemize}}

We set $\d(X):={\rm dist}(X,\T)$ for any $X\subset\R^n$. If $X=\{x\}$ for some $x\in\R^n$, we will simply write $\d(x)$ in place of $\d(\{x\})$. 

\subsection{Admissible trajectories and global asymptotic controllability with regulated cost}

\begin{definition}[Admissible  controls, trajectories, and costs]\label{Admgen}   We say that  $(\alpha,y)$  is an {\em admissible control-trajectory pair}  if there is some $0<S_y\le +\infty$ such that:
	  \begin{itemize}
			\item[(i)]   $\alpha:[0,S_y[\to A$  is Lebesgue measurable; 
			\vsm
			\item[(ii)]  $y:[0,S_y[\to  \R^n\setminus \T$  is a  (Carathéodory)   solution   to the   control system
			\begin{equation}\label{control_sys}
				\dot y(s)= \sum_{i=1}^m  f_i(y(s))\, \alpha^i(s), 
			\end{equation}
		
satisfying,  if $S_y<+\infty$, $\lim_{s\to S_y^-}\d(y(s))=0$.  
	\end{itemize}
	Given an admissible pair $(\alpha,y)$,  we refer to $\mathfrak{I}$ given by
	\begin{equation}\label{Pgen}
				\mathfrak{I}(s):= \int_ 0^{s }  l(y(\sigma),\alpha(\sigma))\, d\sigma, \quad \forall   s\in[0,S_y[
			\end{equation}
	 as the {\it integral cost}, and to $(\alpha,y,\mathfrak{I})$  as an {\em admissible control-trajectory-cost triple}.	For every $x\in\R^n \setminus \T$,  we call   $(\alpha,y)$,  $(\alpha,y,\mathfrak{I})$ as above with $y(0)=x$,  an {\em admissible pair  from $x$} and an {\em admissible triple  from $x$}, respectively.
	\vsm		
	For any admissible pair or triple   such that  $S_y<+\infty$, we  extend   $\alpha$, $y$, and	$\mathfrak{I}$  to $\R_{\geq0}$   by setting $\alpha(s):=\bar a$,  $\bar a\in A$ arbitrary,  and   $ (y,	\mathfrak{I})(s):= \lim_{\sigma\to S_y^-} (y(\sigma),	\mathfrak{I}(\sigma)),$  for any $s\geq S_y$.\footnote{Thanks to hypotheses {\bf (H1)}, {\bf (H2)}, this limit always exists.}
\end{definition}


Let us recall the notion of global asymptotic controllability with regulated cost,  as formulated in \cite{FMR1}.
 \begin{definition}[Global asymptotic controllability with regulated cost] \label{GAC_costo}
	We say that system \eqref{control_sys} is  \textit{globally asymptotically controllable} (in short, GAC) {\em to $\T$} if,  for any   $0<r<R$,    for every $x\in \R^n$ with $\d(x)\leq R$ there exists an admissible  control-trajectory pair $(\alpha,y)$ from $x$ that satisfies the following conditions (i)--(iii):
	$$
	\begin{array}{lllll}
		
		&{\rm (i)} &{\bf d}(y(s)) \leq {\bf \Gamma}(R)\qquad &\forall s \geq 0; \qquad\qquad\qquad &\text{({\it Overshoot boundedness})}\\[1.5ex]
		&{\rm (ii)} &{\bf d}(y(s)) \leq r \qquad &\forall s \geq {\bf S}(R,r);  &\text{({\it Uniform attractiveness})}\\[1.5ex]
		&{\rm (iii)} &\ds\lim_{s \to +\infty} {\bf d}(y(s)) = 0, &\, &\text{({\it Total attractiveness})}
	\end{array}
	$$
	where    $\Gamma:\R_{\ge0}\to\R_{\ge0}$ is a function  with $\Gamma(0)=0$ and $\mathbf{S}:\R_{>0}^2\to \R_{>0}$.	If, moreover, there exists a  function ${\bf W}:\overline{\R^n\setminus\T}\to\R_{\geq0}$ continuous, proper, and positive definite, such that the admissible control-trajectory-cost triple $(\alpha,y,\mathfrak{I})$ associated with  $(\alpha,y)$ above,  satisfies
	$$
	\begin{array}{l}
		\qquad{\rm (iv)} \,\,\, \ds \int_0^{S_y} l(\alpha(s),y(s)) \, ds \leq {\bf W}(x) \,\,\,\qquad\qquad \,\, \text{({\it Uniform cost boundedness})}
	\end{array}
	$$
	($S_y\le +\infty$   as in Def. \ref{Admgen}), we say that  system  \eqref{control_sys} is  \textit{globally asymptotically controllable to $\T$ with} ${\bf W}$-{\em regulated cost} (or simply, {\em with regulated cost}).
\end{definition}


 \subsection{Degree-k feedback generator} \label{def_control}
Let us introduce the sets of admissible bracket pairs associated with the vector fields $\pm f_1,\dots,  \pm f_m$  in the dynamics.  
 
\begin{definition}[Control label] \label{Lie_algebra} For any integer $h$ such that $1\le h\le k$,   let us define the set 
$\F^{(h)}$ of {\it control labels of degree $\leq h$} as 
\[ 
\F^{(h)} := \left\{ 
(B, \mathbf{g},\sgn) \ \left| 
\begin{array}{l} 
\ \sgn\in\{+,-\}  \text{ and $(B, \mathbf{g})$  is an admissible bracket} \\
 \text{ pair of degree $l_B\le h$ such that $\mathbf{g}:=(g_1,\ldots,g_q)$} \\
 \text{  satisfies $g_j\in \{f_1,\dots,f_m\}$ for any $j=1,\dots,q$}
\end{array}
\right. 
\right\}.
\]
We will call {\em degree} and {\em switch number  of a control label} $(B, \mathbf{g},\sgn)\in \F^{(h)}$,  the degree and the switch number of $B$, respectively.  
 \end{definition}
 With any control label  in $\F^{(k)}$ we associate an {\em oriented control}: 
	\begin{definition}[Oriented control]\label{orcon} Consider a time $t>0$ and  $(B, \mathbf{g},+), (B, \mathbf{g},-)  \in \F^{(k)}$.  
		  We define  the corresponding {\em oriented controls} ${\alpha}_{(B,{\bf g},+),t}$,  ${\alpha}_{(B,{\bf g},-),t}$, respectively, by means of the following recursive procedure:
	 	\begin{itemize} 
			\item[(i)] if $\ell_B=1$, i.e. $B=X_j$   for some integer $j\geq 1$,  we set
			$${\alpha}_{(B,{\bf g},+),t}(s):= e_i \qquad  \text{for any $s\in [0,t]$,}
		$$
		where   $i\in\{1,\dots,m\}$  is such that $B(\mathbf{g}) = f_i$   (i.e. $g_j=f_i$);
		\vsm
		\item[(ii)] if $\ell_B\geq 1 $,  we set $ {\alpha}_{(B,{\bf g},-),t}(s):=  -{\alpha}_{(B,{\bf g},+),t}(t-s)$ for any $s\in [0,t]$;
		\vsm
			\item[(iii)] if $\ell_B\ge2$ and $B=[B_1,B_2]$, we set $\mathfrak{s}_1:=\mathfrak{s}_{B_1}$,  $\mathfrak{s}_2:=\mathfrak{s}_{B_2}$, and  \newline $\mathfrak{s}:=\mathfrak{s}_B(=2\mathfrak{s}_1+2\mathfrak{s}_2)$ and,  for any $s\in [0,t]$, we posit	
				\end{itemize} 
\bel{definizione_control}
{\alpha}_{(B,{\bf g},+),t}(s):=
\begin{cases}
{\alpha}_{(B_1,{\bf g},+),\frac{\mathfrak{s}_1}{\mathfrak{s}}t}(s)   \quad  &\text{if }s\in \left[0,  \frac{\mathfrak{s}_1}{\mathfrak{s}}t \right[ \vspace{0.1cm}\\ 
{\alpha}_{(B_2,{\bf g},+),\frac{\mathfrak{s}_2}{\mathfrak{s}}t} \left(s - \frac{\mathfrak{s}_1}{\mathfrak{s}}t \right)   \quad &\text{if }s\in \left[\frac{\mathfrak{s}_1}{\mathfrak{s}}t,  \frac{\mathfrak{s}_1 + \mathfrak{s}_2}{\mathfrak{s}}t \right[ \vspace{0.1cm}\\
{\alpha}_{(B_1,{\bf g},-),\frac{\mathfrak{s}_1}{\mathfrak{s}}t}   \left(s -  \frac{\mathfrak{s}_1 + \mathfrak{s}_2}{\mathfrak{s}}t \right) \,\,\, &\text{if }s\in \left[ \frac{\mathfrak{s}_1 + \mathfrak{s}_2}{\mathfrak{s}}t,  \frac{2\mathfrak{s}_1 + \mathfrak{s}_2}{\mathfrak{s}}t \right[\vspace{0.1cm} \\
{\alpha}_{ (B_2,{\bf g},-),\frac{\mathfrak{s}_2}{\mathfrak{s}}t}  \left(s - \frac{ 2\mathfrak{s}_1 + \mathfrak{s}_2}{\mathfrak{s}}t \right)  \,  &\text{if }s\in \left[ \frac{ 2\mathfrak{s}_1 + \mathfrak{s}_2}{\mathfrak{s}}t, t \right].
\end{cases}
\eeq 
	
	\end{definition}

	\begin{example} 
If $B=[[X_3,X_4],[X_5,X_6]]$, $\mathbf{g}=(f_3,f_2,f_1,f_2,f_3,f_4,f_2,f_3)$ and $t>0$ one has 	  
 \[
{\alpha}_{(B,{\bf g},+),t}(s)\,\,=
\begin{cases}
\begin{aligned}
& e_1 \qquad &\text{if } s\in [ 0, t/16 [ \cup [ 9t/16, 10t/16  [   \\
& e_2 \qquad &\text{if } s\in[ t/16 , 2t/16  [  \cup [ 8t/16, 9t/16  [  \\
& e_3 \qquad & \text{if } s\in[4t/16 ,5t/16  [  \cup [ 13t/16, 14t/16  [  \\
& e_4 \qquad & \text{if } s\in[5t/16 ,6t/16  [  \cup [ 12t/16, 13t/16  [  \\
& -e_1 \qquad &\text{if } s\in[2t/16,3t/16[   \cup  [11t/16, 12t/16 [\\
& -e_2 \qquad &\text{if } s\in[ 3t/16 ,4t/16[  \cup [10t/16,11 t/16 ]  \\
& -e_3 \qquad & \text{if } s\in[6t/16 ,7 t/16[  \cup [15t/16,t] \\
& -e_4 \qquad & \text{if } s\in[7t/16 ,8t/16  [  \cup [ 14t/16, 15t/16  [
\end{aligned}
\end{cases} 
\]
(and ${\alpha}_{(B,{\bf g},-),t}$ can be obtained according to Def. \ref{Lie_algebra},\,(ii)).
\end{example}
 	
 Let us recall a crucial  formula  of  Lie bracket approximation  in \cite[Thm. 3.7]{FR2}, in a form proved in \cite[Lemma 3.1]{Fu}. 
 
\begin{lemma}\label{L_asym}
Assume {\bf (H1)}-{\bf (H2)} and fix $\tilde R>0$. 
Then,   there exist $\bar\delta>0$ and $\w>0$ such that for any $x \in \overline{\B(\T,\tilde R)\setminus\T} $,  any control label $(B, \mathbf{g},\sgn)\in \F^{(k)}$ of degree  $\ell$ and switch number $\mathfrak{s}$,   and any   $t\in[0,\bar\delta]$, there exists a    (unique)   solution $y(\cdot)$  to the Cauchy problem
$$
\dot y(s) =\displaystyle\sum_{i=1}^m f_i(y(s)){\alpha}_{(B,{\bf g},\sgn),t}^i(s), \qquad
 y(0) = x,
$$ 
  defined on the whole interval $[0,t]$ and   satisfying\footnote{We use the notation `$\sgn \, B( \mathbf{g})$' to mean `$+  B( \mathbf{g})$' if $\sgn=+$ and  `$- B( \mathbf{g})$' if $\sgn=-$.}
\bel{asymptotic_formula-}
y(s)\in \B(\T,2\tilde R) \ \ \forall s\in [0,t],  \quad \left| y(t) - x -\sgn \, B( \mathbf{g})(x) \left( \frac{t}{\mathfrak{s}} \right)^\ell  \right| \leq \w\,  t \left(\frac{t}{\mathfrak{s}}\right)^\ell.
\eeq 
\end{lemma}

\begin{remark}\label{R_as_f} In general,  $\bar\delta$  and  $\w$ in the above lemma    do  depend   on $\tilde R$, while  they are independent of $\tilde R$ as soon as  the   hypothesis  $f_1,\dots,f_m\in C^{k-1,1}_b(\R^n)$ replaces the assumption on $f_1,\dots,f_m$ in {\bf (H2)}.
\end{remark}

	\begin{definition}[Degree-$k$ feedback generator] \label{feedback}
		We call {\em degree-$k$ feedback generator} any map $\V:\R^n\setminus\T\to \F^{(k)}$  and write  
		 \bel{les}
\begin{array}{l}
\quad\qquad x\mapsto  \mathcal{V}(x):=(B_x,\mathbf{g}_x,\sgn_x),   \\[1.5ex]
 \ell(x) := \ell_{_{B_x}}, \qquad \mathfrak{s}(x) :=\mathfrak{s}_{_{B_x}}\quad \forall x\in \R^n\setminus\T.  
 \end{array}
 \eeq

		\end{definition}
\begin{definition}[Multiflow] \label{multiflow} 
	Given a degree-$k$ feedback generator  $\mathcal{V}$, for every  $x\in\R^n\setminus\T$  and $t>0$,  we define  the control
	$	\alpha_{x,t}:\R_{\geq0}\to A$, as
	$$
	\alpha_{x,t}(s):={\alpha}_{\mathcal{V}(x),t}(s\wedge t), \qquad \text{ for any  $s\in\R_{\geq0}$.}
	$$
We will refer to the maximal  solution\footnote{By maximal solution we mean  the solution to \eqref{ytx} defined on the largest subinterval of $\R_{\ge0}$ which contains $0$.}    to the Cauchy problem 
	\bel{ytx}
	\dot y(s)= \sum_{i=1}^m f_i\big(y(s)\big)\alpha^i_{x,t}(s), \qquad
		y(0)=x,   
		\eeq 
		as the  {\em $\mathcal{V}$-multiflow starting from $x$ up to the time $t$}  (or simply   {\em $\mathcal{V}$-multiflow}) and will be denoted by $y_{x,t}$.\footnote{By Lemma \ref{L_asym}, given $\tilde R>0$  there is  some $\bar \delta>0$ such that, for all $x\in\R^n\setminus\T$ with  $\d(x)\le \tilde R$ and all $t\in[0,\bar \delta]$,   the trajectory  $y_{x,t}$ is defined on the whole interval $[0,t]$.} 
\end{definition}

\subsection{
Sampling processes}
 We call $\pi:=\{ s_j \}_{j}$    {\em a partition of $\R_{\ge0}$} if $s_0=0$, $s_j  < s_{j+1}$ for any $j\in\N$,  and $\ds\lim_{j\to+\infty}s_j=+\infty$. The {\em sampling time}, or   {\em diameter,} of $\pi$  is defined as the supremum of the set  $\{s_{j+1}-s_j: \ \ j\in\N\}$.
 
\begin{definition}[$\V$-sampling  process]\label{k_traj}\mmR{.}   
Given a degree-$k$ feedback generator $\V$,   we refer to 
$(x, \pi, \alpha_x^\pi, y_x^\pi)$ as a {\em $\V$-sampling process}  if $x\in\R^n\setminus\T$, $\pi:=\{s_j\}_j$ is a partition of $\R_{\geq0}$,   $y_x^{\pi}$ is a continuous function  taking values in $\R^n\setminus\T$  defined  recursively by  
\bel{traj_generalizzate}
\begin{cases}
y_x^{\pi}(s) := y_{x_j,t_j}(s-s_{j-1}) \qquad \text{ for all   $s\in [s_{j-1}, \sigma_j[$ \, and \, $1\leq j \leq\J$} \\
y_x^{\pi}(0) = x,
\end{cases}
\eeq
where,  for all $j\geq 1$, $y_{x_j,t_j}$ is a $\V$-multiflow with $t_j:= s_j-s_{j-1}$, 
  $x_j:= y_x^{\pi}(s_{j-1})$ for all $1\leq j \leq \J$,  and\footnote{We mean $\J=+\infty$    if the set is empty.}
$$
\begin{array}{l}
 \sigma_j :=  \sup\left\{ \sigma\ge s_{j-1}  \text{ : } y_{x_j, t_j} \text{ defined on } [s_{j-1},\sigma[, \  y_{x_j, t_j}([s_{j-1},\sigma[)\subset\R^n\setminus\T \right\},\\[1.0ex]
\J :=\inf\{j: \ \sigma_j\le s_j\}.
\end{array}
$$
We will refer to the map $y_x^\pi:[0, \sigma_{\J }[\to \R^n\setminus\T$  as  a {\em $\V$-sampling trajectory}.
 According to Def. \ref{multiflow}, the  corresponding  {\em $\V$-sampling control} $\alpha_{x}^{\pi}$ is defined as
\[
\alpha_x^{\pi}(s) := \alpha_{x_j, t_j}(s - s_{j-1}) \qquad \text{for all }s\in[s_{j-1}, s_j[\,\cap\,[0, \sigma_{\J }[, \quad 1\leq j\le \J.
\]
Furthermore, we define the {\em $\V$-sampling cost} $\mathfrak{I}_x^{\pi}$ as
\begin{equation}\label{Scostgen}
		\mathfrak{I}_x^{\pi}(s):=\int_0^s   l(y_x^{\pi}(\sigma),\alpha_x^{\pi}(\sigma))\, d\sigma,  \quad \forall s\in[0, \sigma_{\J }[,
			\end{equation}
and we call $(x, \pi, \alpha_x^\pi, y_x^\pi, \mathfrak{I}_x^\pi)$  a {\em $\V$-sampling process-cost}.
\end{definition}
 If  $(\alpha_x^{\pi},y_x^{\pi})$ [resp. $(\alpha_x^{\pi},y_x^{\pi}, \mathfrak{I}_x^{\pi})$] is an admissible pair [resp. triple] from $x$,    we say that the $\V$-sampling process $(x, \pi, \alpha_x^\pi, y_x^\pi)$ [resp. $\V$-sampling process-cost $(x, \pi, \alpha_x^\pi, y_x^\pi,\mathfrak{I}_x^{\pi})$] is {\em admissible}. In this case,
 when $ \sigma_{\J }=S_{y_x^\pi}<+\infty$, we extend $\alpha_x^{\pi}$, $y_x^{\pi}$, and $\mathfrak{I}_x^{\pi}$ to $\R_{\geq0}$,  according to  Def. \ref{Admgen}.

\vsm

In the definition of degree-$k$ sample stabilizability with regulated cost below, we will   consider only $\V$-sampling processes belonging to the subclass of   {\it$\del$-scaled $\V$-sampling processes}, defined as follows. 


\begin{definition}   [$\del$-scaled $\V$-sampling process-cost] \label{vconsistent}
	Given  $\del:=(\delta_1,\ldots,\delta_k)$ in $\R^k_{>0}$, referred to as a {\it multirank},  and a degree-$k$ feedback generator $\V$, we say that $(x, \pi, \alpha_x^\pi, y_x^\pi, \mathfrak{I}_x^\pi)$ [resp. $(x, \pi, \alpha_x^\pi, y_x^\pi)$] is a {\em $\del$-scaled $\V$-sampling process-cost}  [resp. $\del$-scaled $\V$-sampling process] provided  it is  a $\V$-sampling  process-cost [resp. $\V$-sampling process] such that the partition $\pi=\{s_j\}_j$ satisfies
		\bel{vconsistent_part}
		\Delta(k) \, \delta_{\ell_j} \leq s_j - s_{j-1}
		 \leq \delta_{\ell_j} \qquad \forall j\in\N, \ j\ge 1,
		\eeq
 where   $\ell_j := \ell(y_x^\pi (s_{j-1}))$ (see Def. \ref{les})  and  $\Delta(k):= \frac{k-1}{k}$.
\end{definition}

\begin{remark}\label{multirank}
When $k=1$,  the degree-$1$ feedback generator $\V$ takes values in $\F^{(1)}$, so that, in view of Def. \ref{orcon},  a $\V$-sampling process $(x, \pi, \alpha_x^\pi, y_x^\pi)$ is nothing but a standard $\pi$-sampling process  associated with the  feedback law $\alpha(x):=\sgn_x\,e_i$, as soon as $\V(x)=(B_x,\mathbf{g}_x,\sgn_x)$ and $B_x(\mathbf{g}_x) = f_i$.    
In particular, in this case  $\Delta(1)=0$,  $\del=\delta>0$, and  a $\del$-scaled $\V$-sampling process $(x, \pi, \alpha_x^\pi, y_x^\pi)$ coincides with a standard $\pi$-sampling process  such that  diam$(\pi)<\delta$,  exactly as required  in the classical notion of sample stabilizability  (see, for instance, \cite{CLSS,CLRS}).
  When $k>1$,  the notion of  $\del$-scaled $\V$-sampling process  is more restrictive. In particular,  \eqref{vconsistent_part} prescribes for each such process $(x, \pi, \alpha_x^\pi, y_x^\pi)$  both an upper and a lower bound  on the amplitude  $s_j - s_{j-1}$ of every sampling interval,  depending on the degree $\ell_j\in\{1,\dots,k\}$ of $\V(y_x^\pi (s_{j-1}))$. Considering this subclass of processes will actually be crucial in the proof of Thm. \ref{Th_strong}. Indeed, when    a trajectory $y_x^\pi$ defined  on an interval $[s_{j-1},s_{j}]$   approximates the direction of a degree-$\ell_j$ Lie bracket  of length $\ell_j>1$ for the time $t_j:=s_j - s_{j-1}\,(<1)$, the displacement of $y_x^\pi$ is proportional to $t_j^{\ell_j}<t_j$    ---see the  asymptotic formula \eqref{asymptotic_formula-}. Hence,  the lower bound in condition \eqref{vconsistent_part} 
 ensures that for each sampling trajectory the sum of the displacements is divergent. This is  necessary   in order  to build stabilizing sampling  trajectories that uniformly  approach any fixed neighborhood of the target  while providing an  uniform upper bound for the cost  (see Def. \ref{samplestab_costo} below). Incidentally,  notice that  $\Delta(k)$ might be replaced by any function  null for $k=1$,  positive and smaller than 1 otherwise. 	
\end{remark}

\subsection{Degree-$k$  sample stabilizability with  regulated cost}
To state the notion of  {\it degree-$k$  sample stabilizability of system \eqref{control_sys} to $\T$ with  regulated cost}, we first need to provide the following definition. 

%
\begin{definition}[Integral-cost-bound function]\label{ib}  
We say that a function \newline ${\bf \Psi}:  \R_{>0}^3  \to\R_{\ge0}$ is an  {\it integral-cost-bound function} if
$$
{\bf \Psi}(R,v_1,v_2):=\Lambda (R)\cdot\Psi(v_1,v_2),
$$
where 
\begin{itemize}
\item[(i)]  $\Lambda:\R_{>0}\to\R_{>0}$   is  a continuous, increasing function and  $\Lambda\equiv 1$ if $k=1$;
\vsm
\item[(ii)]  $\Psi:  \R_{>0}^2  \to\R_{\ge0}$ is a continuous map, which is 
increasing and unbounded  in the first variable and   decreasing in the second variable;
\vsm
\item[(iii)]   there exists a  strictly decreasing bilateral sequence $(u_i)_{i\in\Z}\subset\R_{>0}$, such that, for some (hence, for any)  $j\in\Z$, one has
	\bel{ppsi}  \sum_{i=j}^{+\infty}\Psi(u_{i},u_{i+1})<+\infty,\quad\text{and} \lim_{i\to-\infty} u_i  = +\infty,  \quad  \lim_{i\to+\infty} u_i  = 0.  
	\eeq
	\end{itemize}	
\end{definition}

\begin{definition}[Degree-$k$ sample stabilizability with regulated cost]
\label{samplestab_costo}
 Let $\V$ be  a  degree-$k$ feedback generator and let $U:\overline{\R^n\setminus\T}\to \R_{\geq0}$ be a  continuous, proper,   and positive definite function.  We  say that $\V$ {\it degree-$k$  $U$-sample stabilizes system \eqref{control_sys} to  $\T$} if 
there exists a multirank map $\del:\R_{>0}^2\to \R_{>0}^k$  such that, for any $0<r<R$,   every $\del(R,r)$-scaled $\V$-sampling process $(x, \pi, \alpha_x^\pi, y_x^\pi)$ with $\d(x)\leq R$ is admissible and satisfies
		$$\begin{array}{l}
		 \text{(i)} \ {\bf d}(y_x^{\pi}(s)) \leq \Gamma(R)\quad \forall s \geq 0,
		\\[1.5ex] 
		 \text{(ii)} \
			 \mathbf{t}(y_x^\pi,r):=\inf\big\{s\geq0 \text{ : } U(y_x^\pi(s))\leq {\varphi}(r) \big\}\leq {\bf T}(R,r),
											\\[1.5ex] 
	 \text{(iii)} \ \text{if $\exists\tau>0$ such that} \ U(y_x^\pi(\tau))\leq {\varphi(r)}, \  \text{then}   \  \d(y_x^\pi(s))\leq r \  \ \forall s\geq\tau,
		\end{array}$$
		where $\Gamma:\R_{\ge0}\to\R_{\ge0}$,   $\varphi:\R_{\geq0}\to\R_{\geq0}$ are continuous, strictly increasing and unbounded    functions  with $\Gamma(0)=0$, $\varphi(0)=0$,  and $\mathbf{T}:\R_{>0}^2\to \R_{>0}$ is a function increasing in the first variable and decreasing in the second one.
	We will refer to property (i) as {\em Overshoot boundedness}  and to (ii)-(iii) as {\em $U$-Uniform attractiveness}. 
	
		If, in addition, there exists an 
  integral-cost-bound function
${\bf \Psi}: \R_{>0}^3\to \R_{ \ge0} $   such  
	 that the $\V$-sampling process-cost $(x, \pi, \alpha_x^\pi, y_x^\pi,\mathfrak{I}_x^\pi)$ associated with the $\V$-sampling process $(x, \pi, \alpha_x^\pi, y_x^\pi)$ above  satisfies the inequality
	 $$
\text{(iv)}\,\, \ds \mathfrak{I}_x^\pi(\mathbf{t}(y_x^\pi,r)) = 	\int_0^{\mathbf{t}(y_x^\pi,r)}   l(y_x^{\pi}(s),\alpha_x^{\pi}(s))ds   \leq  {\bf \Psi}\Big(R,\, U(x), \,U(y_x^\pi(\mathbf{t}(y_x^\pi,r)))\Big)  
	 $$ 
we say that  $\V$ {\em  degree-$k$  $U$-sample stabilizes system  \eqref{control_sys} to $\T$ with ${\bf \Psi}$-regulated cost}. We will refer to (iv) as {\em Uniform cost  boundedness}.

  When there exist some function $U$,  some {degree}-$k$ feedback generator $\V$ [and an integral-cost-bound function ${\bf \Psi}$] such that  $\V$ degree-$k$ $U$-sample stabilizes  system  \eqref{control_sys} to  $\T$ [with ${\bf \Psi}$-regulated cost], we  say that  system  \eqref{control_sys} is {\em degree-$k$ $U$-sample stabilizable to $\T$}  [{\em with ${\bf \Psi}$-regulated cost}]. Sometimes, we will simply say that system  \eqref{control_sys} is {\em degree-$k$  sample stabilizable to $\T$}  [{\em with regulated cost}].\end{definition}

 \begin{remark}
 \label{R_imp}
 By Thm. \ref{Th_strong} below, the existence of a degree-$k$ MRF $U$ implies the notion of  degree-$k$ $U$-sample stabilizability  with regulated cost in  Def. \ref{samplestab_costo},  that might  look quite involved if  compared to classical stabilizability   concepts  (without cost).  
However, we think this is not the case, above all in view of the following facts:
\begin{itemize}
\item[(i)]
 {\em degree-$k$ sample stabilizability  with regulated cost  implies global asymptotic controllability  with regulated cost as in Def. \ref{GAC_costo}} (exactly as classical sample stabilizability implies global asymptotic controllability); and
 \vsm
 \item[(ii)] in the absence of a cost, {\em   system  \eqref{control_sys} is degree-$k$  sample stabilizable to $\T$   for  some $k\ge1$ in the sense of  Def. \ref{samplestab_costo} if and only if it is sample stabilizable to $\T$ in the classical sense of \cite[Def. I.3]{CLSS}, which is in turn equivalent to be degree-$k$  sample stabilizable to $\T$ according to \cite[Def. 2.18]{Fu}.}

\item[(iii)] {\em degree-$1$ sample stabilizability with regulated cost in the sense of Def.   \ref{samplestab_costo} implies   sample stabilizability with regulated cost as defined in \cite{LM}.}
\end{itemize}

 The proofs of statements (i) and (ii) can be found in \cite{FMR1}.
 Moreover,  in view of Rem.  \ref{multirank} and using the notations of Def. \ref{samplestab_costo},  statement  (iii) follows from the following two considerations: first, $U$-uniform attractiveness immediately implies  the standard uniform attractiveness condition 
$$
\exists\, {\bf S}(R,r)>0 \quad \text{such that} \quad  \d(y_x^\pi(s)) \leq r \qquad \text{for all $s\geq {\bf S}(R,r)$}
$$
(with ${\bf S}(R,r)\le {\bf T}(R,r)$),   which in turn characterizes   classical  sample stabilizability (see e.g. \cite{CLSS});  secondly,  for $k=1$  the  uniform cost boundedness condition  (iv) in Def.   \ref{samplestab_costo} implies the cost  bound condition considered in  \cite{LM}, namely
  $$\begin{array}{c}\ds
\int_0^{{\bf s}(y_x^\pi,r)} l(y_x^{\pi}(s),\alpha_x^{\pi}(s))\, ds  \leq \int_0^{{\bf t}(y_x^\pi,r)} l(y_x^{\pi}(s),\alpha_x^{\pi}(s))\, ds \\[2.5ex]
\ \qquad\qquad\qquad\qquad\qquad\qquad\qquad\qquad\le \Psi(U(x),\varphi( r))\le \Psi(U(x),0)=:W(x),
\end{array}$$
where ${\bf s}(y_x^\pi,r):= \inf\{s\geq0 \text{ : } \d(y_x^\pi(\sigma))\leq r \ \text{ for all }\sigma\geq s\}$ \footnote{ Notice that  the $U$-uniform attractiveness   implies  ${\bf s}(y_x^\pi,r)\le   {\bf t}(y_x^\pi,r)$ }.  
\end{remark}

\section{The main result} \label{sec_results}
Together with the notion of {\it degree-$k$ Hamiltonian}, in  this section we provide our most important result: it states that the  existence of a  suitably defined solution $U$ 
to a  degree-$k$  Hamilton-Jacobi inequality  is a sufficient condition for the system to be   degree-$k$  sample stabilizable  with  regulated cost.

\subsection{Degree-$k$ Hamilton-Jacobi dissipative  inequality }
For any integer $h$, $1\le h\le k$, and any $(B, \mathbf{g},\sgn)\in \F^{(h)} $ let us define the subset $A(B, \mathbf{g},\sgn)\subseteq A$ of control values
\bel{abg}
A(B, \mathbf{g},\sgn):=
\begin{cases}
\{ e_i\}    &\text{if $B(\mathbf{g}) = f_i$ and $\sgn=+$}, \\[1.5ex]
\{-e_i\}    &\text{if $B(\mathbf{g}) = f_i$ and $\sgn=-$},  \\[1.5ex]
\big\{\pm e_{j_1},\ldots, \pm e_{j_\ell}\big\} 
&\left\{{\begin{array}{l}
\text{if $B=B(X_{c+1},\dots, X_{c+\ell})$, } \\\text{$2\leq \ell\leq h$, and $\mathbf{g}$ is such that} \\ 
\text{$g_{c+r} = f_{j_r}$ for $r=1,\ldots,\ell$.} 
\end{array}}\right.
\end{cases}	
\eeq

\begin{definition}[Unminimized  Hamiltonian] \label{unkHam}
We define 
 the   \textit{unminimized  Hamiltonian} ${\mathcal H}: (\R^n\setminus \T) \times \R^*\times (\R^n)^* \times \F^{(k)}\to\R$, as   
\[
{\mathcal H}\Big(x,p_0,p,(B, \mathbf{g},\sgn)\Big) :=  
   \langle p,\sgn\, B(\mathbf{g})(x) \rangle+ p_0  \max_{a \in  A(B, \mathbf{g},\sgn)} \, l(x,a) 
\]
for any $ \big(x,p_0,p,(B, \mathbf{g},{\sgn})\big) \in  (\R^n\setminus \T) \times\R^* \times (\R^n)^* \times \F^{(k)}$.

\end{definition}
\begin{definition}[Degree-$h$ Hamiltonian] \label{kHam}
	Given a continuous increasing function  $p_0:\R_{\ge 0}\to [0,1]$ and  an integer $h$, $1\le h\le k$,  we define
   the \textit{{\it degree}-$h$ Hamiltonian} 
   $$
   H[p_0]^{(h)}:(\R^n\setminus \T) \times (\R^n)^*\times\R \to \R,
   $$
 by setting, for every $(x,p,u) \in 
 (\R^n\setminus \T) \times (\R^n)^*\times\R$, 
$$
	H[p_0]^{(h)}(x,p,u) := \min_{(B, \mathbf{g},\sgn) \in \F^{(h)}}{\mathcal H}\Big(x,p_0(u),p,(B, \mathbf{g},\sgn)\Big).
$$
\end{definition}
Notice that minimum exists because the set 
of control labels of degree $\leq h$ is finite.
	Furthermore, under the standing hypotheses,   {\it degree}-$h$ Hamiltonians $H[p_0]^{(h)}$ are well defined and continuous for every $h\in\{1,\dots,k\}$.
Observe also that  
\bel{Hksub}
H[p_0]^{(k)} \leq H[p_0]^{(k-1)} \leq \dots \leq H[p_0]^{(1)},
\eeq
where  the  {\it degree}-1 Hamiltonian $H[p_0]^{(1)}$ reduces to  
$$
H[p_0]^{(1)}(x,p,u)=\min_{a\in A} \left\{\left\langle p,  \sum_{i=1}^m f_i(x)a^i\right\rangle+ p_0(u)\, l(x,a)\right\}. 
$$ 
As an example, let us consider the degree-$2$ Hamiltonian $H[p_0]^{(2)}$: 
\[
\begin{array}{l} H[p_0]^{(2)}(x,p,u) = H[p_0]^{(1)}(x,p,u)\ \ \wedge \\
 \ds\qquad\qquad\qquad \min_{i,j \in \{1,\dots,m\}} \Big\{ \langle p, [f_i, f_j](x) \rangle + \max_{a\in\{\pm e^i, \pm e^j \} }p_0(u) l(x,a) \Big\}.
\end{array}
\]
\begin{definition}[Degree-$k$  MRF] \label{kmrf}  A continuous  map $U: \overline{\R^n \setminus \T} \to \R$ is said to be a {\em degree-$k$  Minimum Restraint Function} (in short,  {\em degree-$k$  MRF}) if it is  proper,  positive definite, and satisfies the
	{\em HJ dissipative  inequality}  
	\bel{dissipative0}  
	H[p_0]^{(k)}(x,p, U(x))  \leq -\gamma(U(x)) \qquad \forall x \in \R^n\setminus \T, \ \  \forall p\in\partial_{\mathbf{P}}U(x),   
	\eeq
	for some continuous and increasing functions $p_0:\R_{\ge 0}\to [0,1]$ and $\gamma: \R_{\ge0}\to \R_{>0}$, to which we will refer  as the  {\em cost multiplier} and the  {\em dissipative rate}, respectively. Furthermore, we say that $U$ is a {\em degree-$k$ Control Lyapunov Function} (in short, {\em degree-$k$ CLF}) if it is a degree-$k$ MRF with $p_0\equiv0$.
\end{definition}

\begin{remark}\label{1-->k}   From \eqref{Hksub} it follows that for  all $q_1$, $q_2\in\N$, $1\le q_1<q_2\le k$, a {\it degree}-$q_1$   MRF (for some $p_0$ and $\gamma$) is also a {\it degree}-$q_2$  MRF  (for the same $p_0$ and $\gamma$), while the converse is false, in general. In particular, a smooth   degree-$k$ MRF for some $k>1$  may exist in situations where there are no smooth {\it degree}-1   MRFs (see the examples in \cite{MR3,MR2}).  Actually,  this is one of the main reasons for considering degree-$k$  MRFs with $k>1$. 
\end{remark}

\begin{remark}\label{CLF}  If  $U$ is   a degree-$k$  MRF for some $p_0$ and $\gamma$,  it is also  a degree-$k$ CLF. Indeed, since $p_0$ and the lagrangian $l$  are nonnegative, 
from the dissipative inequality \eqref{dissipative0} it follows that, for every $x \in \R^n\setminus \T$ and $p\in\partial_{\mathbf{P}}U(x)$, 
	\bel{clf}  
	H[0]^{(k)}(x,p,U(x))=\min_{(B, \mathbf{g},\sgn) \in \F^{(k)}}  \, \langle p,B(\mathbf{g})(x) \rangle  
	\leq -\gamma(U(x)).   
	\eeq
A notion of (locally semiconcave) degree-$k$ CLF was first introduced in \cite{MR2}. The definition of  degree-$k$  MRF  has been anticipated in   \cite{MR3}, in the special case of constant $p_0$ and lagrangian $l$ independent of the control  $a\in A$. 
 \end{remark}

\subsection{Main result}\label{mainre}
To prove that the existence of a degree-$k$ MRF $U$ implies degree-$k$ sample stabilizability with regulated cost,  we need additional assumptions. These conditions
include some integrability requirements on the cost multiplier $p_0$ and on the dissipative rate $\gamma$ and,   in case $k>1$, also the following  {\it $\nu$-semiconcavity} property for $U$, in a neighborhood of the target. 

\begin{definition}[$\nu$-semiconcavity] \label{nu_prop} 
	Let $\M\subset\R^n$ be a  nonempty subset and let $\nu\in[0,1]$.
	We say that a continuous function  $U:\overline{\R^n\setminus\T} \to\R$   {\em is {\it $\nu$-semiconcave on $\M\setminus\T$}}  if 
	there are some positive constants $L$,  $\theta$ and  $C>0$  such that  for every $x$, $\hat x\in  \M\setminus\T$ with  ${\rm sgm}(\hat x,x)\subset\M\setminus\T$ and $|\hat x-x|\le \theta$,  one has
	\bel{scv1}
	\begin{array}{l}
\ds	U(\hat x)-U(x)\leq \langle p, \hat x -x\rangle+ \frac{C}{\d({\rm sgm}(\hat x,x))^\nu} \, |\hat x-x|^2 \quad \forall p\in \partial U(x), \\[1.0ex]
	|p|\le L  \qquad \forall p\in \partial U(x).
	\end{array}
	\eeq
\end{definition}
 

\begin{remark}\label{Rsc0} If $U$  is a locally semiconcave  function on $\R^n\setminus\T$, then, 
	in view of  the property \eqref{semiconcave_est} of its subdifferential,  $U$  is $\nu$-semiconcave on $\M\setminus\T$  with $\nu=0$ for every compact set  $\M\subset\R^n\setminus\T$.    
 More generally,   the notion of  $\nu$-semiconcavity is an extension  of a condition concerning the  distance function $\d$ (from a closed set $\T$)  (see e.g \cite{Ma,MaRi}).  In particular, (see e.g. \cite{CS}) note that 
 	\begin{enumerate}
		\item    if   $\T$ has {\it boundary of class $C^{1,1}$}, then the distance $\d$ is  semiconcave in $\overline{\R^n\setminus \T}$  and    turns out to be $\nu$-semiconcave on $\R^n\setminus \T$ with $\nu=0$, $L=1$,  for any $\theta>0$ and for some  $C>0$.
		
		\item  If    $\T$ satisfies the {\it internal sphere condition  of radius $r>0$}, namely, for all $x\in\T$ there exists $\bar x\in\T$ such that $x\in \B(\bar x,r)\subset\T$,  then the distance $\d$ is  $\nu$-semiconcave on $\R^n\setminus\T$  with $\nu=0$   and satisfies \eqref{scv1} for $L=1$, and $C=1/r$,   for every $\theta>0$. 
		
		\item   If $\T$ {\it is a singleton},  then the distance $\d$  is  $\nu$-semiconcave in $\R^n\setminus\T$   with $\nu=1$  and \eqref{scv1}   holds for $L=C=1$,  for  every $\theta>0$.
	\end{enumerate} 
\end{remark}

We will also use   the following  hypothesis.
{\em  
\begin{itemize}
 \item[{\bf (H3)}]  Let $U$ be a  degree-$k$  MRF and let $p_0$, $\gamma$ be the associated cost multiplier and dissipative rate, respectively. 
 \begin{itemize}
  \item[{\rm (i)}]  If $k>1$, assume that, for some  $\nu \in[0,1]$ and $c>0$, $U$ is $\nu$-semiconcave  on  $\B(\T,c)\setminus\T$,  and that the map $\Theta: \R_{>0}\to \R_{>0}$, defined by
  \end{itemize}
	\bel{Theta}
\Theta(w):= \frac{1}{p_0(w)} \, \vee \, \frac{1}{p_0(w) w^{1-  k^{^{-1}}}} \, \vee \, \frac{1}{p_0(w) \gamma(w)^{k-1}} \, \vee \, \frac{1}{p_0(w) [w^\nu \gamma(w)]^{1-k^{^{-1}}}},
	\eeq
	\begin{itemize}
	\item[] 
	is integrable on $[0,u]$ for any $u>0$.
 \item[{\rm (ii)}] If $k=1$, assume that the map $\Theta$ above, i.e. $\ds\Theta(w)= \frac{1}{p_0(w)}$ 
  is integrable on $[0,u]$ for any $u>0$.
 \end{itemize} 
\end{itemize}}
\begin{remark}\label{RInt-a}  In some   situations, the integrability condition \eqref{Theta} above can be improved. For instance, assume that,  for some $M>0$,  $0<l(x,a)\le M$ for all $(x,a)\in(\R^n\setminus\T)\times A$. Then, given a degree-$k$   MRF  $U$ with cost multiplier $p_0$ and dissipative rate $\gamma$,  set
$$
\lambda(u):= \inf_{\{(x,a)\in\R^n\times A: \ U(x)\ge u\}}\,l(x,a)  \quad \forall u>0
$$
and  consider the  strictly increasing functions  
\bel{tildeg}
  \tilde\gamma(u):=\frac{1}{2}\left(p_0(u)\,\lambda(u) +\gamma(u)\right), \qquad   \tilde p_0(u):=\frac{1}{2}\left(p_0+\frac{\gamma}{M}\right).
\eeq
From \eqref{dissipative0} it follows that $U$ also satisfies  the HJ dissipative  inequality:  
$$
 H[\tilde p_0]^{(k)}(x,p, U(x))  \leq -\tilde\gamma(U(x)) \qquad \forall x \in \R^n\setminus \T, \ \  \forall p\in\partial_{\mathbf{P}}U(x).
$$
Hence, $U$ can be regarded as a degree-$k$ MRF with cost multiplier $\tilde p_0$ and  dissipative rate $\tilde\gamma$. Notice that condition \eqref{Theta} referred to $\tilde p_0$ and $\tilde\gamma$ is weaker than the one corresponding to $p_0$ and $\gamma$.    \end{remark}

Befor stating our main result, let us finally introduce a stronger, global  version of hypothesis {\bf (H2)}.
{\em  
\begin{itemize}
 \item[{\bf (H2)$^*$}] For any $a\in A$,  the map  $x\mapsto l(x,a)$   is  Lipschitz continuous on $\R^n$. Furthermore,  the vector fields $f_1,\dots, f_m$ belong to   $C^{k-1,1}_{b}(\R^n)$.  
 \end{itemize}}

\begin{theorem}[Main result]\label{Th_strong}     
 Assume hypotheses {\bf (H1)}-{\bf (H2)}  and let  $U$ be a degree-$k$ MRF,   which  we suppose to be locally semiconcave on $\R^n\setminus\T$. Then, there exists a degree-$k$ feedback generator $\V$ such that the following statements hold:
\begin{itemize}
\item[{\rm (i)}]   $\V$  degree-$k$  $U$-sample stabilizes   system \eqref{control_sys}  to $\T$.

\item[{\rm (ii)}] If  in addition  $U$ satisfies hypothesis  {\bf (H3)}, then   $\V$  degree-$k$  $U$-sample stabilizes system 
\eqref{control_sys}  to $\T$ with ${\bf \Psi}$-regulated cost,   the map   $$(R,v_1,v_2)\mapsto {\bf \Psi}(R,v_1,v_2)=\Lambda(R)\,\Psi(v_1,v_2)$$  being an integral-cost-bound function,  where $\Psi:\R_{>0}^2\to\R_{\ge0}$   is defined as  
\bel{Psi_Th}
\ds{  \Psi }(v_1,v_2) = \left\{\begin{array}{l}
\ds  0  \,\, \vee\,\, \int_{\frac{v_2}{2}}^{v_1}\Theta(w) \,dw
 \quad \qquad\qquad\qquad\text{if $k=1$,} \\[3ex]
\ds   0  \,\, \vee\,\,  \int_{\frac{v_2}{2}}^{v_1}\Theta\Big(\frac{v_2}{v_1}\,{w}\Big)\,dw \ \ \quad\qquad\quad\text{if $k>1$}
 \end{array}\right.
\eeq
($\Theta$ as in \eqref{Theta}).  In particular, in case $k>1$, if  $U$ (satisfies  {\bf (H3)} and) is  semiconcave  and Lipschitz continuous  on  $\R^n\setminus\T$\footnote{We point out that the notion of semiconcavity on $\R^n\setminus\T$ introduced in Subsection \ref{preliminari} implies that, given $\bar c>0$, there exists a semiconcavity constant $\eta$ for $U$ valid on    $\R^n\setminus B(\T,\bar c)$.}  and  hypothesis {\bf (H2)$^*$} is satisfied, then  $\Lambda$ is  constant.
\end{itemize}
\end{theorem}
The proof of the theorem will be given in the next section. 
\vsm
 From Thm. \ref{Th_strong}  and \cite[Thm. 3.1]{FMR1}, 
the existence of a degree-$k$ MRF $U$ as above implies GAC with ${\bf W}$-regulated cost.  More precisely, we have:


\begin{corollary}\label{Cor_W} Assume hypotheses {\bf (H1)}-{\bf (H2)}. Then,  given a degree-$k$  MRF $U$ satisfying the hypotheses of Thm.  \ref{Th_strong},  system  \eqref{control_sys} is  globally asymptotically controllable to $\T$ with {\bf W}-regulated cost,  the map ${\bf W}$ being defined as 
\bel{W_Th}
\ds{\bf W }(x) = \left\{\begin{array}{l}
\ds\int_{0}^{U(x)}\Theta(w) \,dw
\ \ \   \ \qquad\qquad\qquad\qquad\text{if $k=1$,} \\[1.8ex]
\ds \Lambda(\varphi^{-1}(U(x)))\,\int_{0}^{\frac{U(x)}{2}}\Theta(w)\,dw \ \quad\quad\quad\text{if $k>1$,}
 \end{array}\right.
\eeq
where $\Theta$, $\Lambda$ are as in Thm.  \ref{Th_strong}, and $\varphi$ is as in Def. \ref{samplestab_costo}.  
\end{corollary}

 Theorem \ref{Th_strong} and Corollary \ref{Cor_W}  include and extend several previous results on sufficient conditions for  sample stabilizability and GAC with (or without) a regulated cost, as we illustrate in the following remarks.
\begin{remark}[Case $k=1$]\label{Rconfronto1}  If
the cost multiplier $p_0$ is a  positive constant and $k=1$, then the function $\Theta$ as in hypothesis {\bf (H3)} is trivially integrable and  the integral-cost-bound function ${\bf\Psi}$ takes the form
$${\bf\Psi}(R, v_1,v_2)=\Psi(v_1,v_2)=0  \,\, \vee\,\, \frac{1}{p_0}\Big(v_1-\frac{v_2}{2}\Big)\qquad\text{for all \ $(R,v_1,v_2)\in\R_{>0}^3$.}$$
Actually, for $k=1$
 the proof of Theorem \ref{Th_strong} below can be easily
adapted (see Rem. \ref{multirank} and \cite{Fu,FMR1}) to a general control system 
$$
\dot y=F(y,a), \qquad a\in A\subset\R^m,
$$  
with $A$ nonempty and compact and    $F$ continuous in both variables and locally Lipschitz continuous in $x$, uniformly w.r.t. the control.  Hence,  in view of  Rem. \ref{R_imp},\,(iii),    from  Thm. \ref{Th_strong} and Cor.  \ref{Cor_W} we regain the results on sample stabilizability with regulated cost and on GAC with regulated cost obtained  in \cite{LM} and \cite{MR}, respectively  (actually, with a slightly  sharper bound ong the cost). 
\end{remark}

 \begin{remark}[Case $k>1$]\label{Rconfrontok}  
 Point (i) of  Thm. \ref{Th_strong} (which does not concern a cost),  coincides with the result on (an apparently different notion of)  degree-$k$ sample stabilizability in  \cite{Fu}, as the two definitions are  equivalent (see Rem. \ref{R_imp},\,(ii)).  Furthermore,  Cor. \ref{Cor_W} implies the result  in  \cite{MR2}, where the existence of a
 locally semiconcave degree-$k$ CLF  was shown to guarantee GAC.  Finally,  in the general case with a cost,  Cor. \ref{Cor_W} also implies  the result  in   \cite{MR3}, where a sketch is given of the fact that (under slightly  stronger assumptions than those assumed here) the existence of a degree-$k$  MRF  yields GAC with regulated cost.  \end{remark}

\begin{remark}[Degree-$k$ MRFs and STLC]\label{RTmin}   
Allowing $p_0$  to be an increasing function of $u$, on the one hand, significantly improves the estimate on the cost bound function ${\bf W}$.   On the the hand, it allows us to  reformulate well-known Lie algebraic conditions for the small time local controllability (STLC)  of  system \eqref{control_sys} to  $\T$,  requiring the distance function to be a degree-$k$ MRF.  More specifically, let us consider the case $l\equiv1$, that is, the minimum time problem, and suppose that  
 the distance function $\d$ is a degree-$k$  MRF  for some $p_0$, $\gamma$, and  $k\ge1$, for which {\bf (H3)} is valid. Then,  for all $x\in\R^n\setminus\T$,  $\d$  satisfies the  HJ dissipative  inequality  \eqref{dissipative0}, which   takes now the form
\bel{Tdiss}
\min_{(B, \mathbf{g},\sgn) \in \F^{(k)}}  \, \langle p,\sgn\, B(\mathbf{g})(x) \rangle 
	\leq -\tilde\gamma(\d(x))\qquad \forall p\in\partial_{\mathbf{P}} \d(x),  
\eeq
where $\tilde\gamma(r):=p_0(r)+\gamma(r)$.  It is easy to recognize that this condition, combined with the integrability  assumption in hypothesis {\bf (H3)}, leads back to well-known  higher order weak Petrov (i.e. Lie algebraic)  conditions,  sufficient for the STLC of  the driftless control-affine system \eqref{control_sys}  to the closed target $\T$. By the expression ``weak", we mean    dissipative inequalities as \eqref{Tdiss},  in which the dissipative  rate $\tilde\gamma$ can be 0 at 0, to distinguish them from the classical higher order Petrov conditions, in which $\tilde\gamma$ can be replaced by a positive constant.  In particular,  for any $R>0$ by \eqref{W_Th} we get the following estimate for the minimum time function $T$: 
 $$
T(x)\le \left\{\begin{array}{l}
\ds\int_{0}^{\d(x)}\Theta(w) \,dw
\ \ \   \ \qquad\qquad\text{if $k=1$,} \\[3ex]
\ds \bar\Lambda \,\int_{0}^{\frac{\d(x)}{2}}\Theta(w)\,dw \ \qquad\quad\quad\text{if $k>1$}
 \end{array}\right.  $$
for every $x\in B(\T,R)$, for a suitable constant $ \bar{\bf \Lambda}>0$. In view of the definition \eqref{Theta} of $\Theta$,\footnote{For $U$ considered  as a degree-$k$  MRF with cost multiplier $\tilde p_0$ and  dissipative rate $\tilde\gamma$,  as in Rem. \ref{RInt-a}.} this result is entirely in line with well-known one (see e.g.  \cite{CS,BCD,Ma,MaRi,Kras,KrasQ}, and references therein).  In conclusion, our degree-$k$ sample stabilizability sufficient conditions include as a special case most of the sufficient conditions for STLC of system \eqref{control_sys} to an arbitrary closed set  $\T$ in the literature. We point out that, considering only $p_0\equiv \bar p_0$ positive constant, we would have $\tilde \gamma\ge \bar p_0>0$, so our conditions would include just  ordinary, i.e. non-weak,  higher order Petrov conditions.
\end{remark}

\section{Proof of Theorem  \ref{Th_strong}} \label{proof_section}
Let us begin by proving statement (ii) of the thesis, in case $k>1$.  Let $U: \overline{\R^n \setminus \T} \to \R$ be a degree-$k$ MRF  for some  cost multiplier function $p_0$ and some dissipative rate $\gamma$. Furthermore, assume that 
 $U$ is locally semiconcave on $\R^n\setminus\T$ and satisfies hypothesis {\bf (H3)}, the latter meaning that  $U$  is $\nu$-semiconcave  on  $B(\T,c)\setminus\T$  for some  $\nu \in[0,1]$ and $c>0$,  and that $\Theta$ defined as in \eqref{Theta}  is integrable on $[0,u]$ for every $u>0$.   
  
 \subsection{ A degree-$k$ feedback generator and some preliminary estimates}
%
Let us first establish how the function $U$ is used to build a degree-$k$ feedback generator.  
 Note that  in the HJ dissipative inequality \eqref{dissipative0} we can replace the proximal subdifferential  with the  limiting subdifferential, namely  $U$ satisfies  
\bel{dissipative0lim}
   H[p_0]^{(k)}(x,p, U(x))  \leq -\gamma(U(x))\qquad \forall x\in \R^n\setminus\T, \ \ \forall p\in \partial U(x),
\eeq
   since $U$ is  locally semiconcave, thus  locally Lipschitz continuous, and  $H[p_0]^{(k)}(\cdot)$ is continuous.    
 \begin{definition}[Degree-$k$ $U$-feedback generator] \label{feedback_mrf}
Given $U$ as specified above, choose an arbitrary selection $p(x)\in\partial U(x)$ for any $x\in\R^n\setminus\T$. Then, a degree-$k$ feedback generator $\V:\R^n\setminus\T\to \F^{(k)}$
 is said to be a {\em degree-$k$   $U$-feedback generator}  if 
		\bel{k_feedback}
			{\mathcal H} \Big(x,p_0(U(x)), p(x),\V(x)\Big)\leq -\gamma(U(x)) \quad \text{for all $x\in\R^n\setminus\T$.} 
		\eeq	
	\end{definition}
\noindent Clearly, given a selection $p(x)\in\partial U(x)$,   a  degree-$k$   $U$-feedback generator $\V$ always exists and is defined as a selection 
		$$\V(x)\in  \underset{(B, \mathbf{g},\sgn)\in \F^{(k)}}{\operatorname{argmin}} \,\,{\mathcal H} \Big(x,p_0(U(x)), p(x),(B, \mathbf{g},\sgn)\Big)\qquad \forall x\in \R^n\setminus\T.$$ 

 \begin{remark}\label{Rselection}  Let us point out that, in order to define a degree-$k$   $U$-feedback generator $\V$ as above  (and hence, to prove the thesis of Thm.  \ref{Th_strong}) it is enough to suppose the existence of a function $U$ enjoying all the properties above, except that it satisfies the HJ dissipative   inequality in  \eqref{dissipative0lim} just for an arbitrary selection $p(x)\in\partial U(x)$.
 \end{remark}
 
From now on,  let  a selection  $p(x)\in\partial  U(x)$ and  an associated degree-$k$ $U$-feedback generator $\V(x)=(B_x,\mathbf{g}_x,\sgn_x)$ be given. 
\vsmm 
\noindent Let us define  two $U$-dependent distance-like functions $d_{U_-}$, $d_{U_+}:\R_{\ge0} \to \R_{\ge0}$ as  the smallest distance of the target from the superlevels $\{{\rm U}\geq u\}$ and the largest distance of the target from the sublevels $\{{\rm U}\leq u\}$, respectively. Namely,  for every  $ u\geq 0 $ we set
\bel{zeta}	
\begin{aligned}
	&{d_{U_-}}(u):= \inf \left\{ {\bf d}(x) \text{ : } x\in \overline{\R^n\setminus\T} \ \ {\rm with} \ \ {\rm U}(x) \geq u \right\}	\\ 	
	&{d_{U_+}}(u):= \sup\left\{ {\bf d}(x) \text{ : } x\in \overline{\R^n\setminus\T} \ \ {\rm with} \ \ {\rm U}(x) \leq u \right\} .	
\end{aligned}
\eeq
It is immediate to see that  ${d_{U_-}} $, ${d_{U_+}} $ are strictly increasing   and satisfy
\bel{Lzeta}
\begin{split}
&{d_{U_-}}(0)=\lim_{u\to0^+}{d_{U_-}}(u)=\lim_{u\to0^+}{d_{U_+}}(u)=0={d_{U_+}}(0),\\
&{d_{U_-}}({\rm U}(x))\le \d(x)\le   {d_{U_+}}({\rm U}(x)) \qquad \forall x\in \overline{\R^n\setminus\T}.
\end{split}
\eeq 
Furthermore,    $d_{U_-}$, $d_{U_+}$ can be approximated from below and from above, respectively, by continuous, strictly increasing functions 
 that satisfy \eqref{Lzeta}. In the following, we will still use $d_{U_-}$, $d_{U_+}$ to denote such continuous approximations.  Fix $r$, $R>0$ such that $r<R$ and set
\bel{Nhatn}
\hat{\bf{U}}_R:= {d_{U_-}^{\,\,\,-1}}(R),  \qquad\qquad\tilde R:=d_{U_+}(\hat{\bf{U}}_R) = d_{U_+}\circ {d_{U_-}^{\,\,\,-1}}(R) ,
\eeq
so that
$B(\T, R )\subseteq U^{-1}([0, \hat{\bf U}_R])\subseteq B(\T, \tilde R ).$  Now, by applying Lemma \ref{L_asym}
  for this $\tilde R$,  
  we obtain that there exist some $\bar\delta=\bar\delta(\tilde R)$ and $\omega=\omega(\tilde R)>0$ such that, for any $x\in U^{-1}(]0,\hat{\bf{U}}_R])$ and   $t\in[0,\bar\delta]$, each $\V$-multiflow $y_{x,t}$ is defined on $[0,t]$ and satisfies \eqref{asymptotic_formula-}, i.e.
  \bel{asymptotic_formula-dim}
y_{x,t}([0,t])\subset \B(\T,2\tilde R), \quad
 \left| y_{x,t}(t) - x - \sgn_x\, B_x( \mathbf{g}_x)(x) \left( \frac{t}{\mathfrak{s}} \right)^\ell  \right| \leq \w\,  t \left(\frac{t}{\mathfrak{s}}\right)^\ell,
\eeq 
where 
 $\ell=\ell(x)=\ell(B_x)$,  $\mathfrak{s}=\mathfrak{s}(x) = \mathfrak{s}(B_x)$,    as in  \eqref{les}.
%

\noindent 
From the $\nu$-semiconcavity of $U$ on $B(\T,c)\setminus\T$ and the local semiconcavity of $U$ on $\R^n\setminus\T$ (which implies local Lipschitz continuity),  it  follows that, for the same $\tilde R$ as above, there exist  $\bar C=\bar C(\tilde R)>0$ 
and $L_U=L_U(\tilde R)>0$ such that, 
for every $x$, $\hat x\in  \B(\T,2\tilde R)\setminus\T$ with  ${\rm sgm}(\hat x,x)\subset \B(\T,2\tilde R)\setminus\T$ and $|\hat x-x|\le \theta$  ($\theta$ as in Def. \ref{nu_prop}), one has, for every  $ p\in \partial U(x)$ (see  \eqref{semiconcave_est} and \eqref{scv1})

	\bel{scv1_dim}
	\begin{array}{l}
	U(\hat x)-U(x)\leq \langle p, \hat x -x\rangle+ \bar C\left(1\vee \frac{1}{\d({\rm sgm}(\hat x,x))^\nu}\right) \, |\hat x-x|^2, \\[1.0ex]
	|p|\le L_U.
	\end{array}
	\eeq
 Finally,   under hypothesis {\bf (H2)}  there are some   $M=M(\tilde R)>0$ and $L_l=L_l(\tilde R)>0$, such that, for all $\hat x$, $x\in  \B(\T,2\tilde R)$,  one has
\bel{cost_dim}
|B({\bf g})(x)| \le M \quad \forall (B,{\bf g}, \sgn)\in\F^{(k)},  \quad  |l(\hat x,a)-l(x,a)|\le L_l|\hat x-x| \quad \forall a\in A.
\eeq
For brevity, we  often omit to explicitly write the dependence of the constants $\bar\delta$, $\w$, $L_U$, $M$,  and $L_l$ (and of the constants derived from them)  on $\tilde R$.
\vsm
\noindent
\subsection{Estimating $U$ increments when $\V(x)$ has degree $\ell\le k$}  
Define
\bel{Nhatn}
\hat{\bf{u}}_r:=\chi^{-1}\big({d_{U_+}^{\,\,\,-1}}(r)\big),
\eeq
where the  map $\chi$ is defined by setting, 
 for every $ u\geq 0$,
\bel{lambdachi}
\lambda(u):= u \,\, \vee \,\, u^{\frac{1}{k}}, \qquad 	\quad \chi(u) := u + 2 \lambda(u).
\eeq
 Notice that both  $\lambda$ and  $\chi:\R_{\geq 0}\to\R_{\geq 0}$ are continuous, strictly increasing,  surjective, 
  $\chi(0)=\lambda(0)= 0$, and $\chi(u) > u $ for all $u>0$. 
By construction, we immediately get 
$$
\T \subset U^{-1}([0, \hat{\bf u}_r]) \subseteq B(\T, r )\subset B(\T, R )\subseteq U^{-1}([0, \hat{\bf U}_R])\subseteq B(\T, \tilde R )\subset B(\T, 2\tilde R) .
$$
As a consequence of  Lemma \ref{lemma_potdecrescentecosto} below, the  degree-$k$   MRF  $U$ is decreasing when evaluated along  any multiflow $y_{x,t}$ of the  degree-$k$  $U$-feedback generator $\V$, with $x\in U^{-1}([\hat{\bf u}_r, \hat{\bf U}_R])$ and  $t$ in an interval which depends on the degree  $\ell=\ell(x)$ defined as above. 

\begin{lemma} \label{lemma_potdecrescentecosto} 
 Fix an integer  $\ell\in\{1,\dots,k\}$. Then, using the above notations, we obtain that there exists some positive $\delta_\ell(r)$ ($=\delta_\ell(\tilde R,r)$, i.e. depending on $\tilde R$ as well)
  such that, for any  $x\in U^{-1}([\hat{\bf u}_r, \hat{\bf U}_R])$ verifying  $\ell(x)=\ell$, and any  $t\in[0,\delta_\ell(r)]$, the $\V$-multiflow $y_{x,t}$ verifies
	\bel{decrescenza_costo}
	U(y_{x,t}(t)) - U(x) +\frac{t^{\ell-1}}{\mathfrak{s}^{\ell}} p_0(U(x)) \int_0^t l (y_{x,t}(s), \alpha_{x,t}(s))\,ds \,  \leq -\frac{\gamma(U(x))}{2} \frac{t^{\ell}}{\mathfrak{s}^{\ell}},
	\eeq
 $$\d(y_{x,t}(s))\le 2\tilde R \qquad \forall s\in[0,t],\,\footnote{Incidentally, note that there might exist $s<t$ such that  $y_{x,t}(s)\in\T$.}$$ 
and  \bel{sopra}\ds
\frac{\hat{\bf u}_r}{2}\le  U(y_{x,t}(t))< U(x).
\eeq 
\end{lemma}



\begin{proof}  
	 We  begin by proving  property \eqref{sopra}.
	  Let us set
	\bel{tilde_delta} 	\delta_0 :=  1\,\, \wedge \,\, \bar\delta \,\, \wedge \,\,  \frac{\theta}{M},     \qquad 
	\hat\delta_\ell(r):=  \frac{\hat{\bf u}_r^{\frac{1}{\ell}}}{(2 (M+\w) L_{_U})^{\frac{1}{\ell}} \vee L_{_U}M} .
	\eeq 
	Fix $x\in U^{-1}([\hat{\bf u}_r, \hat{\bf U}_R])$ with   degree $\ell(x)=\ell$.  By the first relation in \eqref{asymptotic_formula-dim}, for every $t\in[0,\delta_0]$ one has  $y_{x,t}(s)\in \B(\T,2\tilde R)$ for all $s\in[0,t]$. Hence, in view of  \eqref{cost_dim} and \eqref{tilde_delta},  one has
	$|y_{x,t}(s)-x|\le \theta$ for all $s\in[0,t]$. Furthermore, the definition of $L_U$ and the fact that $U\equiv 0$ on $\T$ imply that $U(x)\le L_U \d(x)$. Therefore, since $\mathfrak{s}\geq 1$,  for all times $t\in[0,\delta_0\land \hat\delta_\ell(r)]$  the second relation in \eqref{asymptotic_formula-dim}  yields 
	\bel{first_est}
	|y_{x,t}(t) -x| \leq (M+\w t) \left( \frac{t}{\mathfrak{s}} \right)^\ell \leq (M+\w) t^{\ell} \leq \frac{\hat{\bf u}_r}{2L_{_U}} \leq \frac{U(x)}{2L_{_U}} \leq \frac{{\bf d}(x)}{2}.
	\eeq
	As a first consequence, one has  $|U(y_{x,t}(t)) - U(x)| \leq \displaystyle \frac{\hat{\bf u}_r}{2}$. Since $U(x) \geq \hat{\bf u}_r$, this  proves the left-hand side of \eqref{sopra}. 
	
	 Now let us prove   \eqref{decrescenza_costo}. Since for any $z\in \B(x,\d(x)/2)$ one has $\d(z)\geq \d(x)/2$, and \eqref{first_est} implies that  ${\rm sgm}(x, y_{x,t}(t)))\subset\B( x,\d(x)/2)$, one has
		\bel{stima_dist}
	\d({\rm sgm}(x, y_{x,t}(t)))\geq 
	\frac{{\bf d}(x)}{2} \geq \frac{\hat{\bf u}_r}{2 L_{_U}}.
	\eeq
Let now $\check\delta_\ell(r)$ be the unique solution of the equation\footnote{The solution's uniqueness  follows from the trivial fact that the function of $\delta$ on the left-hand side of \eqref{eq_delta} is strictly increasing and unbounded.}
	\bel{eq_delta}
	\left( L_{_U}\w + \frac{L_l M}{2} \right) \delta + \frac{\bar C(2L_{_U})^\nu(M+\w)^2}{[(2L_U)\land \hat{\bf u}_r]^\nu} \delta^\ell =\frac{\gamma(\hat{\bf u}_r)}{2},
	\eeq
where ($M$ and) $L_l$ are as in  \eqref{cost_dim},  and set	\bel{delta_h}
	\delta_\ell(r):= \delta_0 \,\,\wedge \,\,\hat\delta_\ell(r)\,\, \wedge\,\,\check\delta_\ell(r).
	\eeq
 Using \eqref{k_feedback}, \eqref{scv1_dim}, \eqref{stima_dist}, \eqref{asymptotic_formula-dim}, for any $t\in[0, \delta_\ell(r)]$,  we get\,\footnote{In view of hypothesis {\bf (H2)}, for any  subset $A'\subseteq A$ the function $x\mapsto  \max_{a\in A'} l(x,a)$ is $L_l$-Lipschitz continuous on $\B(\T,2\tilde R)\setminus\T$. Recall also that $p_0(\cdot)$ and  $t$  are $\le 1$.} 

	\[
	\begin{array}{l}
	\ds	U(y_{x,t}(t)) - U(x)  + \left(\frac{t^{{\ell}-1}}{\mathfrak{s}^{\ell}}\right)p_0(U(x)) \int_0^t l(y_{x,t}(s), \alpha_{x,t}(s))\,ds \,  \\[2ex]
	\qquad\leq \langle p(x), y_{x,t}(t) -x\rangle  +\bar C\vee\frac{ \bar C}{\d({\rm sgm}(x, y_{x,t}(t)))^\nu}\,|y_{x,t}(t) -x|^2  \\[2ex]
	\ds\qquad\qquad\qquad\qquad+\left(\frac{t^{{\ell}-1}}{\mathfrak{s}^{\ell}}\right)p_0(U(x)) \int_0^t \max_{a\in A(\V(x))} l(y_{x,t}(s),a)ds \, \\[2ex]
	\qquad \ds	\leq  \left( \frac{t}{\mathfrak{s}} \right)^{\ell} \left[ \langle p(x), \sgn_x B_x(\mathbf{g}_x)(x)\rangle + L_{_U} \w t+ p_0(U(x))\max_{a\in A(\V(x)) } l(x,a) \right. \\[2ex]
\ds\quad\qquad\qquad\qquad\qquad\qquad \left. +\frac{p_0(U(x)) L_l M}{2} t  + \frac{\bar C(2L_{_U})^\nu (M+\w t)^2}{[(2L_U)\land \hat{\bf u}_r]^\nu} \left( \frac{t}{\mathfrak{s}} \right)^{\ell}  \right]    \\[2ex]
\ds		\qquad\leq  \left( \frac{t}{\mathfrak{s}} \right)^{\ell} \left[ -\gamma(U(x)) + \left( L_{_U} \w + \frac{L_l M}{2}  \right) t + \frac{\bar C(2L_{_U})^\nu (M+\w)^2 }{[(2L_U)\land \hat{\bf u}_r]^\nu} t^{\ell}  \right]    \\[2ex]
\ds		\qquad\leq \left( \frac{t}{\mathfrak{s}} \right)^{\ell}\left[-\gamma(U(x)) + \frac{\gamma(\hat{\bf u}_r)}{2} \right] \leq -\frac{\gamma(U(x))}{2} \left(\frac{t}{\mathfrak{s}}\right)^{\ell}.
	\end{array}
	\]
\end{proof}

  In the next lemma we  
determine two $U$-sublevels such that the $\V$-multiflows issuing from them remain in $\B(\T,r)$ until a time that depends on the utilized iterated Lie bracket. 
\begin{lemma}\label{incastro_chi}
 Fix an integer  $\ell\in\{1,\dots,k\}$. Then, for $\lambda$ and $\chi$   as in \eqref{lambdachi} and  using the above notations,   for any  $x\in U^{-1}(]0, \hat{\bf U}_R])$ with the degree $\ell(x)$ of $\V(x)$ equal to $\ell$ and any  $t\in[0,\delta_\ell(r)]$, the $\V$-multiflow $y_{x,t}$ satisfies
	\begin{itemize}
		\item[{\rm (i)}] if $y_{x,t}(\Ss)\in U^{-1}(]0,\hat{\bf u}_r])$ for some $\Ss\in[0,t]$, then $U(y_{x,t}(s))\leq \hat{\bf u}_r +  \lambda(\hat{\bf u}_r)$  for any $s\in[\Ss,t]$, so that, in particular,  $\d(y_{x,t}(s))\leq r$  for any $s\in[\Ss,t]$;
		\item[{\rm (ii)}] if $x\in U^{-1}(]\hat{\bf u}_r, \hat{\bf u}_r + \lambda(\hat{\bf u}_r)])$, then $U(y_{x,t}(s))\leq \chi(\hat{\bf u}_r)$ for any $s\in[0,t]$, so that, in particular, $\d(y_{x,t}(s))\leq r$ for any $s\in[0,t]$.\,\footnote{Notice that $\hat{\bf u}_r< \hat{\bf u}_r +  \lambda(\hat{\bf u}_r)< \hat{\bf U}_R$ by definition.} 
	\end{itemize}		

\end{lemma}


%

\begin{proof}
	The proofs of (i) and (ii) follow the same lines, so we prove (ii) only. Let $x\in U^{-1}(]\hat{\bf u}_r,\hat{\bf u}_r+\lambda(\hat{\bf u}_r)])$. Since $y_{x,t}(s)\in \B(\T,2\tilde R)$ for all $s\in[0,t]$ and, in particular,   $t\le\hat\delta_\ell(r)$ as defined in \eqref{tilde_delta},  recalling the definition of $\lambda$  one has
	\[
	|U(y_{x,t}(s))-U(x)|\leq L_{_U} |y_{x,t}(s)-x| \leq L_{_U} M s \leq \lambda(\hat{\bf u}_r) \qquad \forall s\in[0,t].
	\]
Hence,   $U(y_{x,t}(s))\leq \chi(\hat{\bf u}_r)$ for any $s\in[0,t]$, by  the definition of $\chi$. In view of  \eqref{Lzeta} and  \eqref{Nhatn}, this implies that  $\d(y_{x,t}(s))\leq r$ for any $s\in[0,t]$. 
\end{proof}

\noindent  \subsection{ Stabilizing $\del$-scaled $\V$-sampling processes}  Given $0<r<R$, set 
$$
\del=\del(R,r):=(\delta_1(r), \dots,\delta_k(r)),
$$
where  $\delta_\ell(r)$ is as in \eqref{delta_h} for any $\ell=1,\dots,k$.\footnote{The  multirank $\del$ turns out to depend on $R$ as well, because $\delta_\ell$ do depend on $\tilde R$, which is a function of $R$ by construction, for any $\ell=1\dots,k$.} 
  Let $(x, \pi, \alpha_x^\pi, y_x^\pi)$ be an arbitrary $\del(R,r)$-scaled $\V$-sampling process such that  $\d(x)\leq R$.  Since the partition $\pi=(s_j)_{j\in\N}$ of $\R_{\geq0}$ satisfies  \eqref{vconsistent_part} (and $\B(\T,R)\subseteq U^{-1}(]0,\hat{\bf U}_R])$),  thanks to the definition of $\del$, 
Lemma \ref{lemma_potdecrescentecosto} implies that $(\alpha_x^\pi, y_x^\pi)$ is an admissible  control-trajectory pair from $x$,   and  
\bel{Gamma_pre}
y_x^\pi(s)\in \B(\T,2\tilde R) \qquad\text{ for all $s\ge0$.}
\eeq
 In particular, using the notations of Def. \ref{multiflow} and Def. \ref{k_traj},  for every $j\in\N$,  $1\le j\le \J$,   the $\V$-multiflow  $y_{x_j,t_j}$  is defined on the whole interval $[0,t_j]$, where  
\[
t_j:= s_j-s_{j-1},  \qquad x_1:=x, \qquad  \qquad x_{j+1}:= y_{x_{j},t_{j}}(s_{j}-s_{j-1}).
\]
Hence, one has  $y_x^\pi(s_j)=x_{j+1}$  for all $0\leq j < \J$. (However,  whenever $\J<+\infty$,   the trajectory  $y_x^\pi$ reaches for the first time the target $\T$ at some $\sigma_\J\in]s_{\J-1},s_\J]$.  In this case,  after $\sigma_\J$, the pair $(\alpha_x^\pi, y_x^\pi)$ is extended constantly, while  the $\J$-th $\V$-multiflow remains defined as before, over the entire interval $[s_{\J-1},s_\J]$, so that  we may well have $y_x^\pi(s_\J)=y_x^\pi(\sigma_\J)\ne x_{\J+1}$.)
 \vsmm

In the following lemma we provide an upper bound for the time taken by $y_x^\pi$ to reach the sublevel set $U^{-1}(]0,\hat{\bf u}_r[)$.

\begin{lemma} \label{lemma_hatn} 
Consider $(x, \pi, \alpha_x^\pi, y_x^\pi)$ as above. Then, one has
	\bel{iotar}
	\iota_r :=\inf \{ j \in\N \text{ : } U(y_x^{\pi}(s_j))  < \hat{\bf u}_r \} < +\infty,
	\eeq
 and $\iota_r\leq \J$ when $\J<+\infty$. 
	 Moreover,
	\bel{S}
	s_{\iota_r} \leq {\bf T}(R,r):= \beta(k) \sqrt[k]{\frac{2(\hat{\bf U}_R-\hat{\bf u}_r) J(R,r)^{k-1}}{\gamma(\hat{\bf u}_r)}} +1,
	\eeq
	where  $ \beta(k)$ is as in \eqref{bk} and,  for 
	\bel{muR} 
	\mu(R,r):=\min\{\delta_\ell(r) \text{ : } \ell=1,\dots,k \},
	\eeq
$J(R,r)$ is defined as\footnote{We use $[\![\cdot]\!]$ to denote the integer part.} 
	 
	\bel{J}
	J(R,r):= 
	\begin{cases}
	\left[\!\!\left[ \frac{2 (\hat{\bf U}_R-\hat{\bf u}_r) \beta(k)^k }{\gamma(\hat{\bf u}_r) \mu(R,r)^k \Delta(k)^k}   \right]\!\!\right] +1 \qquad &\text{if $k\geq 2$},  \\
	\qquad\quad 1 &\text{if $k=1$.}
	\end{cases}
	\eeq

\end{lemma}

\begin{proof}
	If $\J<+\infty$, one has by definition that  $U(y_x^{\pi}(s_j))=0$ for every $j\ge\J$. Therefore,   $\iota_r \le \J<+\infty$.   If $\J=+\infty$ and we assume by contradiction that
	$\iota_r=+\infty$, then  
	$U(y_x^{\pi}(s_j)) \geq \hat{\bf u}_r$ for any $j\in\N$. Hence, if we set   $\ell_j:= \ell(y_x^\pi(s_{j-1}))$, and $\mathfrak{s}_j := \mathfrak{s}(y_x^\pi(s_{j-1}))$, by \eqref{decrescenza_costo} and the monotonicity of $\gamma$, for any integer $j\ge1$,  we have
	\bel{passaggi}
	\begin{array}{l}
		\hat{\bf u}_r - \hat{\bf U}_R \leq U(y_x^{\pi}(s_j)) - U(x) = \left[  U(y_x^\pi(s_j)) - U(y_x^\pi(s_{j-1}))\right] + \dots \\
		\ds\qquad\qquad\qquad\qquad\qquad\qquad\qquad\qquad\qquad\qquad\qquad + \left[ U(y_x^\pi(s_1))- U(x)  \right] \ds	\\
		\ds \ \qquad\qquad\leq -\frac{\gamma(U(y_x^\pi(s_{j-1})))}{2}  \left( \frac{t_j}{\mathfrak{s}_j}  \right)^{\ell_j} - \cdots - \frac{\gamma(U(x))}{2}  \left( \frac{t_1}{\mathfrak{s}_1}  \right)^{\ell_1}  \\
		\ds\ \qquad\qquad\leq - \frac{\gamma(\hat{\bf u}_r)}{2 \beta(k)^k}[t_1^{k} + \dots + t_j^{k}] ,
	\end{array}
	\eeq
	where we have used that $1\geq (t_j/\mathfrak{s}_j)^{\ell_j} \geq (t_j/\beta(k))^k$ for any integer $j\ge1$. If $k=1$, then $t_1+\dots+t_j=s_j\to+\infty$ when $j\to+\infty$, by the very definition of partition of $\R_{\geq0}$. Otherwise, if $k\geq2$, \eqref{vconsistent_part} implies that
$t_1^k+\dots+t_j^k \geq \mu(R,r)^k \Delta(k)^k \, j$, which tends  to $+\infty$ as $j\to+\infty$.
Hence, in both cases we reach a contradiction, so that $\iota_r<+\infty$.

Let us now  prove \eqref{S}. Using the Jensen inequality in \eqref{passaggi}, we deduce that
	\bel{jensen}
	\hat{\bf U}_R-\hat{\bf u}_r \geq \frac{\gamma(\hat{\bf u}_r)}{2 \beta(k)^k} \frac{s_j^k}{j^{k-1}} 
	\eeq
for any $j\leq \iota_r-1$. Now, taking $j=\iota_r-1$ and $k=1$ in \eqref{jensen}, we get \eqref{S} in the particular case $k=1$. Indeed, one has 
\[
s_{\iota_r} \leq s_{\iota_r -1} + \delta_1(R,r) \leq  \frac{2(\hat{\bf U}_R - \hat{\bf u}_r)}{\gamma(\hat{\bf u}_r)} +1.
\]
\noindent
Let now $k\geq2$. Again by \eqref{passaggi}, 	as soon as  $j\leq \iota_r-1$ we obtain
	\[
	\hat{\bf U}_R-\hat{\bf u}_r \geq \frac{\gamma(\hat{\bf u}_r)}{2 \beta(k)^k} \mu(R, r)^k\Delta(k)^k\, j.
	\]
 In particular, taking $j=\iota_r-1$ in the previous relation, we deduce that 
	\bel{passaggi2}
	\iota_r\leq J(R,r),
	\eeq
	 with $J(R,r)$  as in \eqref{J}. By \eqref{passaggi2}, \eqref{jensen} we finally obtain \eqref{S} also for $k\geq2$. 
\end{proof}
We are now ready to show  that the degree-$k$  $U$-feedback generator $\V$ degree-$k$ $U$-sample stabilizes  system \eqref{control_sys} to $\T$. Fix an arbitrary $\del(R,r)$-scaled $\V$-sampling process  $(x, \pi, \alpha_x^\pi, y_x^\pi)$   such that  $\d(x)\leq R$, as above. 
\vsmm
\noindent By \eqref{Gamma_pre} and  \eqref{Nhatn}, setting 
\[
\Gamma(R):= 2\tilde R =d_{U_+}({d_{U_-}^{\,\,\,-1}}(R)),
\] 
we directly get the overshoot boundedness property (i)  in Def. \ref{samplestab_costo}, i.e. 
${\bf d}(y_x^{\pi}(s)) \leq \Gamma(R)$   for any $s\geq0.$ 
Furthermore,  the properties of the functions $d_{U_+}$ and $d_{U_-}$ (see \eqref{Lzeta})  imply that $\ds \lim_{R\to0} \Gamma(R)=0$. 
\vsmm
\noindent In order to prove  the $U$-uniform attractiveness  property (ii)-(iii),  observe  that, 
 by the very definition of $\iota_r<+\infty$ in Lemma \ref{lemma_hatn}, one has $U(y_x^\pi(s_{\iota_r}))\leq \hat{\bf u}_r$.  Accordingly, the  time 
\bel{bf_t}
\mathbf{t}= \Ss(y_x^\pi,r):=\inf\{s\geq0 \text{ : } U(y_x^\pi(s))\leq \hat{\bf u}_r \},
\eeq
 (is finite and) satisfies 
\bel{unif_attr}
\mathbf{t} \leq s_{\iota_r} \leq {\bf T}(R,r),
\eeq
for ${\bf T}(R,r)$ as in \eqref{S}.   Set 
\bel{varphi}
{\varphi}(r):=\hat{\bf u}_r \qquad \text{for any $r>0$.}
\eeq
To conclude this step, it only remains to show condition (iii) in Def. \ref{samplestab_costo},  which is equivalent to prove that
\bel{entrap}
\d(y_x^\pi(s)) \leq r \qquad \text{for any $s\geq \mathbf{t}$.}
\eeq
  Let $\bar \j\geq1$ be the integer such that $\mathbf{t}\in [s_{\bar\j-1}, s_{\bar\j}[$, so that,
by Lemma \ref{incastro_chi}, (i), one has that $U(y_x^{\pi}(s))\le \hat{\bf u}_r + \lambda(\hat{\bf u}_r)$ and  $\d(y_x^\pi(s))\leq r$ for any $s\in[\mathbf{t}, s_{\bar \j}]$. Hence, either case (a) or case (b) below occurs: 
\begin{itemize}
	
	
	\item[(a)] $U\left(y_x^{\pi}(s_{\bar \j})\right) \leq \hat{\bf u}_r$. 
	
	\item[(b)] $U(y_x^{\pi}(s_{\bar \j})) \in ]\hat{\bf u}_r, \hat{\bf u}_r + \lambda(\hat{\bf u}_r)]$.
\end{itemize}



\noindent {\em Case} (a). By Lemma \ref{incastro_chi}, (i), it follows that $\d(y_x^{\pi}(s))\leq r$ for all $s \in [s_{\bar\j}, s_{\bar\j+1}]$. Then either case (a) or case (b) above holds with $\bar\j +1$ replacing $\bar\j$.
 \vsmm
\noindent {\em Case} (b). 
Arguing as in Lemma \ref{lemma_hatn}, we deduce that there exists an integer $\iota(\bar\j)$ satisfying
$$
\iota(\bar\j):= \inf \{ j \in \N \text{ : } j\geq \bar\j+1 \ \ \text{ and } \ \ U\left( y_x^{\pi}(s_j)\right) < \hat{\bf u}_r  \}<+\infty.
$$
In particular, by Lemma \ref{lemma_potdecrescentecosto}, the sequence $\left(U(y_x^{\pi}(s_j))\right)_j$ is decreasing for $\bar\j \leq j \leq \iota(\bar\j)$ and by Lemma \ref{incastro_chi}, (ii), we get $\d(y_x^{\pi}(s)) \leq r$ for all $s\in [s_{\bar\j}, s_{\iota(\bar\j)}]$. Moreover, in view of Lemma \ref{incastro_chi}, (i), we also have $\d(y_x^{\pi}(s)) \leq r$ for all $s\in [s_{\iota(\bar\j)}, s_{\iota(\bar\j)+1}]$.
Then, either case (a) or case (b) above holds with $\iota(\bar\j)+1$ replacing $\bar\j$. 

The proof of
 \eqref{entrap} is thus  concluded. In particular, let us point out that the continuity and monotonicity properties of the functions $\Gamma$, $\varphi$ and ${\bf T}$ are straightforward consequence of their very definitions and of the properties of the functions $d_{U_-}$ and $d_{U_+}$.

\noindent
\subsection{ The cost boundedness property }  Let $0<r<R$ and let $(x, \pi, \alpha_x^\pi, y_x^\pi, \mathfrak{I}_x^\pi)$ be an arbitrary $\del(R,r)$-scaled $\V$-sampling process-cost  such that  $\d(x)\leq R$,  associated with a $\del(R,r)$-scaled $\V$-sampling process $(x, \pi, \alpha_x^\pi, y_x^\pi)$ as in the previous step. We use the same notations as above and, in addition, set
$$
 u_j:= U(x_{j })\qquad \text{for any $j\in\N$, \ $1\le j\le\J+1$.}
 $$
 Observe that from the previous lemmas it  follows that
\bel{u_j}
\hat{\bf u}_r
\le u_{\iota_r}<u_{\iota_r-1}<\dots<u_1=U(x), \quad \frac{\hat{\bf u}_r}{2}\le u_{\iota_r+1}<u_{\iota_r}.
\eeq
 To prove  the uniform cost boundedness  property (iv)  in Def. \ref{samplestab_costo}, we  need to construct a integral-cost-bound function ${\bf \Psi}=\Lambda\,\Psi$,  such that 
 $$
   \mathfrak{I}_x^\pi(\mathbf{t}) = \int_0^{\mathbf{t}}   l(y_x^{\pi}(s),\alpha_x^{\pi}(s))\, ds   \leq \Lambda(R)\, \Psi(U(x),\varphi(r)),  
   $$
where  
$\mathbf{t}$ is as in  \eqref{bf_t}.  When  $U(x)\le \varphi(r)$, the integral is zero  (because $\mathbf{t} =0$)  and the estimate is trivial for every nonnegative ${\bf \Psi}$.  Hence, let us suppose $U(x)> \varphi(r)=\hat{\bf u}_r$. 
Since $l\ge0$, from \eqref{unif_attr}, \eqref{u_j} and   Lemma \ref{lemma_potdecrescentecosto},  we get
\bel{est_l}
\begin{split}
\int_0^{\mathbf{t}}&  l( y_x^\pi(s), \alpha_x^\pi(s)) ds \leq \int_0^{s_{\iota_r} \wedge S_{y_x^\pi}}  l( y_x^\pi(s), \alpha_x^\pi(s)) ds   \\
 &= \sum_{j=1}^{\iota_r-1} \int_{s_{j-1}}^{s_{j}} l( y_x^\pi(s), \alpha_x^\pi(s)) ds  + \int_{s_{\iota_r-1}}^{s_{\iota_r} \wedge \sigma_\J}  l( y_x^\pi(s), \alpha_x^\pi(s)) ds   \\
& 	\leq \sum_{j=1}^{\iota_r } \frac{\mathfrak{s}_j^{\ell_j}}{t_j^{\ell_j-1}} \frac{u_j-u_{j+1}}{p_0(u_j)} ,
\end{split}
\eeq
 where  $S_{y_x^\pi}$ is as in Def. \ref{Admgen}, so that $S_{y_x^\pi}=\sigma_\J$, and,  in particular,  
\[ 
\begin{split}
\int_{s_{\iota_r-1}}^{s_{\iota_r} \wedge S_{y_x^\pi}}  l( y_x^\pi(s), \alpha_x^\pi(s)) ds &\leq \int_{s_{\iota_r-1}}^{s_{\iota_r}}   l( y_{x_{\iota_r}, t_{\iota_r}}(s), \alpha_{x_{\iota_r}, t_{\iota_r}}(s)) ds \\
&\leq  \frac{(\mathfrak{s}_{\iota_r})^{\ell_{\iota_r}}(u_{\iota_r}-u_{\iota_r+1})}{(t_{\iota_r})^{\ell_{\iota_r}-1}{p_0(u_{\iota_r})}  }.
\end{split}
\]
Define the   function $\hat \Lambda:\{(v_1,v_2)\in \R_{>0}^2: \  v_2\le v_1\}  \to\R_{>0}$,  given by
$$
\hat \Lambda(v_1,v_2):=\frac{v_1}{v_2} \qquad (\ge 1).
$$
For every $j=1,\dots,\iota_r+1$, recalling that $\frac{\varphi(r)}{2}\le u_j\le U(x)$ , we have
\bel{est_u_j}
\frac{\varphi(r)}{2\hat \Lambda(U(x),\varphi(r))}\le\frac{u_j}{\hat \Lambda(U(x),\varphi(r))}\le  \frac{U(x)}{\hat \Lambda(U(x),\varphi(r))}=\varphi(r). 
\eeq
Now, let us  fix $j\in\{1,\dots,\iota_r+1\}$ and let us    estimate the quantity $\frac{1}{t_j^{\ell_j-1}}$.
   The definition  \eqref{eq_delta}   implies that either
\[
\left(L_{_U}\w + \frac{ L_l M}{2} \right) \check\delta_\ell(r) \geq \frac{\gamma(\varphi(r))}{4} 
\]
or 
\[ 
\frac{\bar C(2L_{_U})^\nu (M+\w)^2}{[(2L_{_U})\land\varphi(r)]^\nu} (\check\delta_\ell(r))^\ell\geq \frac{\gamma(\varphi(r))}{4}.
\] 
Thus,  \eqref{vconsistent_part},  \eqref{est_u_j} yield that  
\bel{tj_alla_hj}
\begin{split}
\frac{1}{t_j^{\ell_j-1}} &\leq \frac{1}{\Delta(k)^{\ell_j-1}} \Bigg( \frac{1}{\delta_0^{\ell_j-1}} \,\, \vee \,\, \frac{1}{(\check\delta_{\ell_j})^{\ell_j-1}} \,\, \vee \,\, \frac{1}{(\hat\delta_{\ell_j})^{\ell_j-1}} \Bigg) \\
&\leq \frac{1}{\Delta(k)^{k-1}} \Bigg[ \frac{1}{\delta_0^{k-1}} \,\, \vee \,\, \frac{4\bar C (2L_{_U})^\nu (M+\w)^2}{\gamma^{1-\frac{1}{k}}\big(\frac{u_j}{\hat \Lambda}\big) \big[(2L_U)\land\frac{u_j}{\hat \Lambda}\big]^{\nu-\frac{\nu}{k}}}  \,\, \vee \,\, \frac{4L_{_U}\w + 2 L_l M}{\gamma^{k-1}\big(\frac{u_j}{\hat \Lambda}\big)}  \\
& \qquad\quad\quad\vee \,\,  \left. \frac{(2L_{_U}(M+\w))^{1-\frac{1}{k}} \, \vee \, (L_{_U}M)^{k-1}}{\Big(\frac{u_j}{\hat \Lambda}\Big)^{1-\frac{1}{k}}} \right] \le \tilde \Lambda(\tilde R) \frac{p_0\big(\frac{u_j}{\hat \Lambda}\big)}{\beta(k)^k} \Theta\Big(\frac{u_j}{\hat \Lambda}\Big),
\end{split}
\eeq
where ($\Theta$ is as in \eqref{Theta} and)  we have  written $\hat \Lambda$ in place of $\hat \Lambda(U(x),\varphi(r))$   and we have  set  
\bel{K}
\begin{array}{l}
\tilde \Lambda(\tilde R) := \frac{\beta(k)^k}{\Delta(k)^{k-1}}\Big(\frac{1}{\delta_0^{k-1}}\,\,\vee\,\,  4\bar C [(2L_{_U})^{\frac{\nu}{k}}\vee(2L_{_U})^\nu] (M+\w)^2  \\
\qquad\qquad\qquad\qquad\vee\,\, (4L_{_U}\w + 2 L_l M)  \,\,\vee\,\,  \left( 2L_{_U}(M+\w)\right)^{1-\frac{1}{k}}    \,\,\vee\,\, (L_{_U}M)^{k-1} \Big).
\end{array}
\eeq
Note that   $\tilde \Lambda(\tilde R)$  depends on $\tilde R$ defined as in \eqref{Nhatn}, likewise all constants $\delta_0$, $\bar C$, $L_U$, $M$, $\w$, and  $L_l$ actually depend  on $\tilde R$.   
  Hence, from \eqref{est_l}, \eqref{est_u_j}, and the monotonicity of $p_0$,   we get
\bel{est_l3}
\begin{array}{l}
\ds\int_0^{\mathbf{t}}  l( y_x^\pi(s), \alpha_x^\pi(s)) ds \leq   \sum_{j=1}^{\iota_r } \frac{\mathfrak{s}_j^{\ell_j}}{t_j^{\ell_j-1}} \frac{u_j-u_{j+1}}{p_0(u_j)}\le \tilde \Lambda(\tilde R)  \sum_{j=1}^{\iota_r }\Theta\Big(\frac{u_j}{\hat \Lambda}\Big)\,(u_j-u_{j+1}) \\
\quad\qquad\qquad\qquad\qquad\le \ds  \tilde \Lambda(\tilde R)  \int_{\frac{ \varphi(r)}{2 }}^{U(x)}\Theta\Big(\frac{w}{\hat \Lambda}\Big)\,dw=  \Lambda(R)\Psi(U(x),\varphi(r)),
\end{array}
\eeq
as soon as we set $\Lambda(R):=\tilde \Lambda(\tilde R)$ ($\tilde R$ is in turn a function of $R$ by \eqref{Nhatn}) and define    $\Psi: \R_{>0}^2  \to \R_{\geq0}$, as  
\bel{Psi}
\ds\Psi (v_1,v_2) :=  0  \,\, \vee\,\, \int_{\frac{v_2}{2}}^{v_1}\Theta\Big(\frac{v_2}{v_1}\cdot{w}\Big)\,dw, \qquad \text{ for any $(v_1,v_2)\in\R_{>0}^2$}.
\eeq
To complete the proof of statement (ii) in case $k>1$, let us show that $\Psi$ is an  integral-cost-bound function.  Given an arbitrary $\bar v>0$, let us  consider the bilateral sequence $(v_i)_{i\in\Z}$,  given by 
\[
v_1=\bar v, \qquad v_{i+1}= \frac{v_{i}}{2} \quad \forall i\in\Z, 
\]
so that $\hat \Lambda(v_i,v_{i+1})=2$ for all $i$. Thanks to   hypothesis {\bf (H3)}, we have
\bel{CONTI}
\begin{array}{l}
\ds	\sum_{i=1}^{+\infty} \Psi(v_i,v_{i+1}) =   \sum_{i=1}^{+\infty}  \int_{v_{i+2}}^{v_i}\Theta\Big(\frac{w}{2}\Big)\,dw  
 =  \sum_{i=1}^{+\infty}  \int_{\frac{\bar v}{2^{i+1}}}^{\frac{\bar v}{2^{i-1}}}\Theta\Big(\frac{w}{2}\Big)\,dw \\
\ds\ \qquad\qquad  \le 2\int_0^{\bar v} \Theta\Big(\frac{w}{2}\Big)\,dw
=4\int_0^{\frac{\bar v}{2}}\Theta(w)\,dw<+\infty.
\end{array}
\eeq
Clearly, $\Psi$ is  continuous and  the required monotonicity and unboundedness  properties are immediate consequence of the definition of $\Theta$.  Similarly, $\Lambda$ is increasing, hence it can be approximated from above by a continuous increasing function, still  denoted by $\Lambda$ with a small abuse of notation.
Finally, when  hypothesis {\bf (H2)$^*$}  
is satisfied and $U$ is also Lipschitz continuous and semiconcave on $\R^n\setminus\T$,  it is immediate to see that $\bar C$, $L_{_U}$, $M$, $\w$ and $L_{_l}$ in \eqref{K} (as well as $\delta_0$ in \eqref{tilde_delta})
 no longer depend on the compact set $B(\T,2\tilde R)$, so that $\Lambda$   turns out to be    constant. 

\subsection{Sketch of the proof of (i) for  $k\ge1$ and of (ii) for $k=1$}

 Let now $U$ be a degree-$k$  MRF  for some   $p_0$ and  $\gamma$, and assume $U$ merely locally semiconcave on $\R^n\setminus\T$.\footnote{Incidentally, with regard to the stabilizability issue, $p_0$ plays no role.}  
\vsm 
 
\noindent   Consider a selection $p(x)\in\partial U(x)$ and let $\V$ be a degree-$k$ $U$-feedback generator. Fix $0<r<R$ and, using the above notations,    define $\hat{\bf U}_R$ and $\tilde R$ as in \eqref{Nhatn}, so that for any $x\in U^{-1}(]0,\hat{\bf{U}}_R])$ and   $t\in[0,\bar\delta]$, each $\V$-multiflow $y_{x,t}$ is defined on $[0,t]$ and satisfies \eqref{asymptotic_formula-dim}.  In particular, all constants   $\bar\delta$, $\w$,   $M$,  and $L_l$ as in  \eqref{asymptotic_formula-dim}, \eqref{cost_dim} are fixed, as above,   on the compact set $\B(\T,2\tilde R)$. The fundamental difference from the previous proof is that the Lipschitz continuity and semiconcavity constants of $U$ can no longer be defined up to the target, but on compact sets $\M\subset\R^n\setminus\T$ only.
\vsmm
\noindent For $d_{U_-}$ and $d_{U_+}$ as in \eqref{zeta},  set 
\bel{hatn2}
\varphi(r)=\check{\bf \u}_r:= \frac{1}{2}\, d_{U_+}^{-1}(r), \qquad
\check{\bf r}_r:=\frac{1}{2}\, d_{U_-}\Big(\frac{\check{\bf u}_r}{2}\Big).
\eeq
Accordingly, it holds 
\[
\begin{split}
\T &\subset  \B(\T, \check{\bf r}_r) \subset \B(\T, 2\check{\bf r}_r) \subseteq U^{-1}\Big(\,\Big]0,\frac{\check{\bf u}_r}{2}\Big]\Big)\subset U^{-1}(]0,\check{\bf u}_r])\subset U^{-1}(]0,2\check{\bf u}_r]) \\
&\subseteq \B(\T,r)\subset \B(\T,R)\subseteq U^{-1}(]0,\hat{\bf U}_R])\subseteq \B(\T,  \tilde R )\subset \B(\T, 2\tilde R).
\end{split}
\]
Let us define  $$\M:=\overline{\B(\T,2\tilde R) \setminus \B(\T,  \check{\bf r}_r)},$$
and let $\eta_{_\M}$, $L_{_\M}>0$ be the semiconcavity and the Lipschitz constant of $U$ on $\M$, respectively.
Then, any $\V$-multiflow $y_{x,t}$ starting from $x\in U^{-1}([\check{\bf \u}_r, \hat{\bf U}_R])$ up to $t\leq \bar\delta\land \frac{\check{\bf r}_r}{M} \land \frac{\check{\bf \u}_r}{2L_{_\M} M}$ is such that 
\bel{segmento}
U(y_{x,t}(t))\geq \frac{\check{\bf u}_r}{2}, \qquad y_{x,t}(s)\in \M \quad \forall s\in[0,t],  \qquad {\rm sgm}(x, y_{x,t}(t)) \subset \M. 
\eeq
Indeed $|U(y_{x,t}(s))-U(x)|\leq L_{_\M} |y_{x,t}(s)-x|\leq  L_{_\M} M s \leq \frac{\check{\bf \u}_r}{2}$ for any $s\in[0,t]$, so that  $U(y_{x,t}(s))\geq  \frac{\check{\bf u}_r}{2}$ for any $s\in[0,t]$. In view of \eqref{Lzeta}, the first two relations in \eqref{segmento} follow. The last property in \eqref{segmento} can be deduced by the facts that $\d(x) \geq 2\check{\bf r}_r$  and $|y_{x,t}(t)-x|\leq M t\leq  \check{\bf r}_r$. Now set
\bel{delta_h2}
\delta(R,r) := 1 \,\,\, \wedge \,\,\, \bar\delta\,\,\, \wedge \,\,\, \frac{\check{\bf \u}_r}{2L_{_\M} M} \,\,\, \wedge \,\,\, \frac{\check{\bf r}_r}{M} \,\,\, \wedge \,\,\, \frac{\gamma(\check{\bf \u}_r)}{2(L_{_\M}\w + L_{_l} M + \eta_{_\M} (M+\w)^2 )},
\eeq
Note that now $\delta(R,r)$ does not depend on $\ell$. 
 Suitably modifying the constants in the proof of Lemma \ref{lemma_potdecrescentecosto}, it is  possible to prove that any $\V$-multiflow $y_{x,t}$ starting from $x\in U^{-1}([\check{\bf u}_r, \hat{\bf U}_R])$ up to time $t\leq \delta(R,r)$ satisfies inequality \eqref{decrescenza_costo}. Indeed, thanks to \eqref{segmento} we can apply \eqref{semiconcave_est} in place of \eqref{scv1} and derive that
	\[
	\begin{split}
		&U(y_{x,t}(t)) - U(x)  + p_0(U(x)) \int_0^t l(y_{x,t}(s), \alpha_{x,t}(s))\,ds \, \left(\frac{t^{\ell-1}}{\mathfrak{s}^\ell}\right) \\
		&\ \leq  \Big( \frac{t}{\mathfrak{s}} \Big)^\ell \Big[ -\gamma(U(x)) + t \Big( L_{_\M} \w + L_{_l} M + \eta_{_\M}(M+\w)^2\Big)    \Big]   \leq -\frac{\gamma(U(x))}{2} \left(\frac{t}{\mathfrak{s}}\right)^\ell.
	\end{split}
	\]
From now on, the proof of (i) is a simple adaptation of the previous one, so we omit it. 

\vsm 

\noindent
For what concerns the cost estimate in the case $k=1$, arguing as above
 we simply get
\bel{est_l2}
\ds\int_0^{\mathbf{t}}  l( y_x^\pi(s), \alpha_x^\pi(s)) ds
 	\leq \sum_{j=1}^{\iota_r } \frac{u_j-u_{j+1}}{p_0(u_j)} \leq 
	 \int_{\frac{\varphi(r)}{2}}^{U(x)}\Theta(w)\,dw
\eeq
where now $\Theta(w)=\frac{1}{p_0(w)}$ for all $w\ge0$. 
 \qed

 \section{Generalizing form and regularity of the dynamics for the case $k>1$}\label{Sec_gen}
 In Remark \ref{Rconfronto1} we observed that, for the case $k=1$,  the extension  of Thm. \ref{Th_strong} to a general  system 
\bel{dyn_gen}
 \dot y=F(y,a), \qquad a\in A,
\eeq
  is straightforward. The case $k>1$  is much more involved.  However, a first easy  extension can be achieved by still considering  a  driftless  control-affine system  but with a control set  of the following form$$\tilde A:=\Big\{\beta_1e_1,\dots, \beta_m e_m, -\gamma_1e_1,\dots,-\gamma_m e_m \ \ \ \ \text{s.t.} \ \ \ \ \beta_i,\gamma_i>0,\,\,\,\forall  i=1,\ldots,m  \Big\}.$$  For instance, the definitions of oriented controls for the case $k=1,2$ should be replaced by the following ones:

 \begin{itemize} 
 	\item[(i)] if $\ell_B=1$, i.e. $B=X_j$   for some integer $j\geq 1$,  we set
 	$${\alpha}_{(X_j,{\bf g},+),t}(s):= \beta_i e^i \qquad  \text{for any $s\in [0,t]$,}
 	$$
 	and 
 	$${\alpha}_{(X_j,{\bf g},-),t}(s):= -\gamma_i e^i \qquad  \text{for any $s\in [0,t]$,}
 	$$
 	where   $i\in\{1,\dots,m\}$  is such that $B(\mathbf{g}) = f_i$   (i.e. $g_j=f_i$);
 	\vsm
 	\item[(ii)] if $\ell_B=2$, $B=[X_j,X_{j+1}]$ and $B(\mathbf{g}) = [f_{i_1},f_{i_2}]$  (namely, $g_{j}=f_{i_1}$ and $g_{j+1}=f_{i_2}$), we set
	\[
	\begin{split}
&{\alpha}_{(B,{\bf g},+),t}(s):= \\
 &\quad \begin{cases}
 		{\alpha}_{(X_j,{\bf g},+),\frac{\tau_1}{\tau}t}(s)  \Big(= \beta_{i_1} e^{i_1}  \Big)   &s\in \left[0,  \frac{\tau_1}{\tau}t \right[ \vspace{0.1cm}\\ 
 		{\alpha}_{(X_{j+1},{\bf g},+),\frac{\tau_2}{\tau}t} \left(s - \frac{\tau_1}{\tau}t\right)  \Big(=\beta_{i_2} e^{i_2}\Big)    &s\in \left[\frac{\tau_1}{\tau}t ,  \frac{\tau_1+\tau_2}{\tau}t \right[ \vspace{0.1cm}\\
 		{\alpha}_{(X_j,{\bf g},-),\frac{\tau_3}{\tau}t}   \left(s -  \frac{\tau_1+\tau_2}{\tau}t\right) \Big(=  -\gamma_{i_1} e^{i_1}\Big)  &s\in \left[ \frac{\tau_1+\tau_2}{\tau}t ,  \frac{\tau_1+\tau_2+\tau_3}{\tau}t \right[\vspace{0.1cm} \\
 		{\alpha}_{ (X_{j+1},{\bf g},-),\frac{\tau_4}{\tau}t } \left(s -  \frac{\tau_1+\tau_2+\tau_3}{\tau}t \right) \Big( =-\gamma_{i_2} e^{i_2}\Big)  &s\in \left[\frac{\tau_1+\tau_2+\tau_3}{\tau}t, t \right]
 	\end{cases}
 \end{split}
 \]
 	and	$$ {\alpha}_{(B,{\bf g},-),t}(s):=  {\alpha}_{(B,{\bf g},+),t}(t-s) \qquad\text{for any $s\in [0,t]$},$$
	where  	$$ \tau_1=1/\beta_{i_1}, \ \  \tau_2=1/\beta_{i_2},  \ \ \tau_3=1/\gamma_{i_1},\ \  \tau_4 =1/\gamma_{i_2}, \ \  \tau=\tau_1+\tau_2+\tau_3+\tau_4.$$ 	
 \end{itemize} 
 It is easy to verify that 	the asymptotic formula \eqref{asymptotic_formula-} should be replaced by the $(i_1,i_2)$-dependent inequality
 	\bel{asymptotic_formula---}
 	\left| y(t) - x -\sgn \, B( \mathbf{g})(x) \left( \frac{t}{\tau} \right)^2  \right| \leq \w\, t \left(\frac{t}{\tau}\right)^2.
 	\eeq 
 For $k>2$ one would proceed in an akin way, so that,   by means of  suitable notions of the degree-$k$  Hamiltonians,  an extension of Thm. \ref{Th_strong} would follow without any other difficulty.
 An extension of the case $k>1$ for   control system \eqref{dyn_gen}
 to a {\it general dynamics} $F$ and a {\it general control set $A$,} requires  strong selection properties for the set-valued map $x\rightsquigarrow F(x,A)$. As a natural example, one could assume the existence of a selection 
 \bel{selection}
 \check F (x, C) \subseteq  F(x,A)
 \eeq
 where $\check F$ is a driftless control-affine dynamics  and $C$ is like $\tilde A$, namely, for some integer $r>1$,
 $$
 \check F(y,c):= 	
 \sum_{i=1}^r  \check f_i(y)\, c^i, \qquad (c^1,\dots,c^r) \in C,
 $$
 ($\check f_1,\ldots,\check f_r$ are vector fields) and
 $$C:=\Big\{\beta_1e_1,\dots,\beta_re_r, -\gamma_1e_1,\dots,-\gamma_r e_r \ \ \text{s.t.} \ \ \beta_i,\gamma_i>0,\,\,\,\forall  i=1,\ldots,r  \Big\}\subset \R^r.$$
 Thanks to   selection  \eqref{selection}, 
 a  degree-$k$ feedback generator for the control-affine system $\dot y  = \sum_{i=1}^r  \check f_i(y )\, c^i$, \, $c\in C$,  defines a feedback law also for the original system $\dot y=F(y,a)$, $a\in A$. Hence, by defining  Hamiltonians $\check H[p_0]^{(h)}$, $h\leq k$, corresponding to the dynamics $\check F$ in a way similar to 
 \eqref{kHam}, we obtain a result like Thm. \ref{Th_strong},  valid for the system $\dot y= \check F(y,c)$, \, $c\in C$,  which in turn provides degree-$k$ sample stabilizability with regulated cost for the original  system  \eqref{dyn_gen}.
 
\vsm
  A non-trivial further generalization might concern  control-affine systems with a non-zero drift. Obviously, the corresponding degree-$k$  Hamiltonian should  be based also on minimizations over sets including suitable brackets containing the drift among their factors (see  \cite[Thm. 5.1]{MR2}).  
\vsm 

  Finally, another  interesting generalization might consist, still for a driftless control-affine system $\dot y  = \sum_{i=1}^m   f_i(y )\, a^i$,   in considering  less regular vector fields $f_1,\ldots,f_m$: if $B$ is a formal  iterated bracket, in view of \cite[Thm. 3.7]{FR2} (see also \cite{FR1}), an asymptotic formula similar to  \eqref{asymptotic_formula-}  still holds provided $ B(f_1,\ldots,f_m)$ is a  $L^\infty$ map.  In such a case one can make use of  a set-valued iterated Lie bracket $ B_{set}(f_1,\ldots,f_m)$, which happens to be upper semi-continuous.  For instance, in the case of the bracket $[f_5,f_7]$ with  $f_5,f_7$ locally Lipschitz continuous, the bracket $[f_5,f_7]$ is an $L^\infty$ map defined almost everywhere. The corresponding set-valued bracket is  defined, for every $x$, as 
  $$
  [f_5,f_7]_{set}(x):= co\{\lim_{n\to\infty}[f_5,f_7](x_n), \,\,\,x_n\to x\},
  $$
   where `co' denotes the convex hull and  the limits are intended for all sequences $(x_n)$  converging to $x$ made of  of differentiability points  for both $f_5$ and $f_7$. This bracket has revealed  fit for extending  various basic results on vector fields' families (see \cite{RS1,RS2}) and might be useful also for an extension of the results of the present paper (see  \cite[Thm. 4.1]{MR2},  and  also \cite{BFS} for the STLC's issue).

\end{document}